\documentclass{article}
\usepackage{amsfonts}
\usepackage{latexsym}
\usepackage{mathrsfs}
\usepackage{amssymb}
\usepackage{amsmath}
\usepackage{amsthm}
\usepackage{indentfirst}
\usepackage{color}
\usepackage{float}
\usepackage{graphicx}

\allowdisplaybreaks
\newtheorem{theorem}{\color{black}\indent Theorem}[section]
\newtheorem{lemma}{\color{black}\indent Lemma}[section]

\newtheorem{definition}{\color{black}\indent Definition}[section]
\newtheorem{remark}{\color{black}\indent Remark}[section]

\begin{document}
\title{\LARGE\bf Existence and nonexistence of solutions to a critical
biharmonic equation with logarithmic perturbation}
\author{Qi Li \qquad Yuzhu Han$^{\dag}$\qquad Tianlong Wang}
 \date{}
 \maketitle

\footnotetext{\hspace{-1.9mm}$^\dag$Corresponding author.\\
Email addresses: yzhan@jlu.edu.cn(Y. Han).

\thanks{
$^*$Supported by the National Key Research and Development Program of China
(grant no.2020YFA0714101).}}
\begin{center}
{\noindent\it\small School of Mathematics, Jilin University,
 Changchun 130012, PR China}
\end{center}

\date{}
\maketitle

{\bf Abstract}\ In this paper, the following critical biharmonic elliptic problem
\begin{eqnarray*}
\begin{cases}
\Delta^2u= \lambda u+\mu u\ln u^2+|u|^{2^{**}-2}u, &x\in\Omega,\\
u=\dfrac{\partial u}{\partial \nu}=0, &x\in\partial\Omega
\end{cases}
\end{eqnarray*}
is considered, where $\Omega\subset \mathbb{R}^{N}$ is a bounded smooth domain with $N\geq5$.
Some interesting phenomenon occurs due to the uncertainty
of the sign of the logarithmic term. It is shown, mainly by using
Mountain Pass Lemma, that the problem admits at lest one nontrivial weak
solution under some appropriate assumptions of $\lambda$ and $\mu$.
Moreover, a nonexistence result is also obtained.
Comparing the results in this paper with the known ones, one sees that
some new phenomena occur when the logarithmic perturbation is introduced.




\section{Introduction}
\setcounter{equation}{0}

In this paper, we are concerned with the following critical biharmonic elliptic problem
\begin{eqnarray}\label{P1}
\begin{cases}
\Delta^2u= \lambda u+\mu u\ln u^2+|u|^{2^{**}-2}u, &x\in\Omega,\\
u=\dfrac{\partial u}{\partial \nu}=0, &x\in\partial\Omega,
\end{cases}
\end{eqnarray}
where $\Omega\subset \mathbb{R}^{N}(N\geq5)$ is a bounded domain with smooth boundary
$\partial\Omega$, $\nu$ is the unit outward normal on $\partial\Omega$,
$2^{**}:=\frac{2N}{N-4}$ is the critical Sobolev exponent for the  embedding
$H_0^2(\Omega)\hookrightarrow L^{2^{**}}(\Omega)$, $\lambda,\mu\in \mathbb{R}$,
and $\Delta^2$ is the biharmonic operator.

Problem \eqref{P1} is closely related to the stationary version of the following equation
\begin{eqnarray}\label{P2}
u_{tt}+\Delta^2u=f(t,x,u),
\end{eqnarray}
which has wide application in describing a variety of phenomena in physics and other
applied sciences. For example, \eqref{P2} can be used to describe the longitudinal motion
of an elastic-plastic bar. Interested reader may also refer to
\cite{MMAlGharabli,LJAn,LJAn2,LJAn3} for more background of problems
like \eqref{P2}.

On the other hand, problem \eqref{P1} is a special from of the following critical elliptic
boundary value problem
\begin{eqnarray}\label{P5}
\begin{cases}
(-\Delta)^{m} u=\lambda u+|u|^{p-2}u+f(x,u), &x\in\Omega,\\
D^{\alpha}u=0,\qquad for\ |\alpha|\leq m-1,&x\in\partial\Omega,
\end{cases}
\end{eqnarray}
where $m\in \mathbb{N}$, $\Omega\subset \mathbb{R}^{N}(N>2m)$ is a bounded smooth domain,
$\lambda\in \mathbb{R}$, $p=\frac{2N}{N-2m}$ and $f(x,u)$ is a nonlinear term.
One of the main features of problem \eqref{P5} is that it involves a critical term in
the sense that the embedding $H_0^m(\Omega)\hookrightarrow L^{p}(\Omega)$ is not compact,
which brings essential difficulty when proving the existence of weak solutions to \eqref{P5}.
For this reason, the study of elliptic problems with critical exponents is
both interesting and challenging. However, it was mainly after the outstanding
work of Br\'{e}zis and Nirenberg \cite{Brezis}
that such problems were extensively studied, and many interesting results were obtained on
the existence, nonexistence and multiplicity of weak solutions to such problems.

In \cite{Brezis}, problem \eqref{P5} with $m=1$, $N\geq3$ and $f(x,u)=0$ was
studied by Br\'{e}zis and Nirenberg, and a new idea for establishing a Palais-Smale
compactness condition for the energy functional was introduced. This makes it possible
to obtain a solution of the critical problem by using Mountain Pass Lemma.
Interestingly, they found that the conditions for the existence of positive solutions
are quite different between $N\geq4$ and $N=3$. When $N\geq4$, it is shown that
problem \eqref{P5} admits a positive solution if $\lambda\in(0,\widetilde{\lambda}_1)$,
where $\widetilde{\lambda}_1$ denotes the first eigenvalue of $-\Delta$ with
homogeneous Dirichlet boundary condition on $\Omega$. When $N=3$ and $\Omega$ is a ball,
problem \eqref{P5} has a positive solution if and only if $\lambda\in(\frac{1}{4}\widetilde{\lambda}_1,\widetilde{\lambda}_1)$.
They also proved that problem has no solution if $\lambda<0$ and $\Omega$ is a star-shaped domain,
which directly follows from Pohozaev's identity (see \cite{PohozaevSI}).
Later, their ideas were borrowed by Gu et al. \cite{YGu} to study higher order critical problems,
i.e. problem \eqref{P5} with $m=2$, $f(x,u)=0$ and $N\geq5$. They showed that
problem \eqref{P5} possesses at least one nontrivial weak solution
if $\lambda\in(0,\lambda_{1}(\Omega))$ and $N\geq8$. When $N =5,6,7$ and $\Omega$ is a ball,
they showed that there exist two positive constants $\lambda^{**}(N)<\lambda^{*}(N)<\lambda_{1}(\Omega)$
such that problem has at least one nontrivial weak solution if $\lambda\in(\lambda^{*}(N),\lambda_{1}(\Omega))$,
and there has no nontrivial solution if $\lambda<\lambda^{**}(N)$.
Here $\lambda_{1}(\Omega)$ denotes the first eigenvalue of $\Delta^{2}$ with homogeneous
Dirichlet boundary condition on $\Omega$. The above mentioned results imply that whether
or not problem \eqref{P5} with $m=1,2$ and $f(x,u)=0$ admits a weak solution depends not
only on the parameter $\lambda$, but also on the space dimension $N$.
It is worth mentioning that there are also many other interesting works about elliptic
problems with critical exponents, among the huge amount of which, we only refer the
interested reader to \cite{COAlves,BBarrios,LDAmbrosio,Deng,Deng4,Deng2,Deng3,DNaimen,XMingqi}
and the references therein.

A natural question is whether or not a lower order perturbation will affect the existence
and nonexistence results for problem \eqref{P5}. Recently, Deng et al. \cite{Deng}
investigated the existence and nonexistence of positive solutions to problem \eqref{P5}
with $m=1$, $f(x,u)=\mu u\ln u^2$ and $\mu\neq0$. After some delicate estimates on the logarithmic
term, they showed that problem \eqref{P5} admits a positive ground state solution,
which is also a mountain pass type solution when $N\geq4$, $\lambda\in \mathbb{R}$ and $\mu>0$.
Comparing this result with the corresponding one in \cite{Brezis} one sees that the range
of $\lambda$ for the problem with $\mu>0$ to admit a positive solution is larger than that with
$\mu=0$, which implies that the logarithmic term has a positive effect for the existence
of positive solutions in this case. When $\mu<0$, the discussion becomes a little more complicated.
However, when $N=3,4$, they also established the existence of at least one positive solution under
some additional assumptions on $\lambda$ and $\mu$. This also shows that the appearance of
logarithmic perturbation leads to some interesting phenomenon, mainly due to the uncertainty
of the sign of $u\ln u^2$.

Inspired mainly by \cite{Deng,YGu}, we study problem \eqref{P1} and consider how the term $u\ln u^2$
affects the existence and nonexistence of nontrivial solutions to problem \eqref{P1}.
It turns out that the results depend heavily on the sign of $\mu$. When $\mu>0$,
we first show that the energy functional satisfies the mountain pass geometry around $0$.
Then we obtain a local compactness condition (known as $(PS)_c$ condition) for the (PS) sequences,
with the help of Br\'{e}zis-Lieb's lemma. Finally, a mountain pass type solution follows
from a standard variational argument. A crucial step in this process is to prove that the
mountain pass level is smaller than a critical number $c(S)$, which is done after
some delicate estimates on various norms of the truncated Talenti functions.
Special attention is paid to the case $N=8$, due to the critical characterization of the integrations.
When $\mu<0$, the situation is different. Although the mountain pass geometry is still satisfied and
the mountain pass level around $0$ is bounded from above by the same critical number,
we cannot prove the $(PS)_c$ condition in general. By applying the Mountain Pass Lemma without
compactness condition, we obtain a nontrivial solution to problem \eqref{P1}.
But we donot know whether or not the solution is of mountain pass type.
Moreover, a nonexistence result is also obtained when $\mu<0$.
Comparing our results with that in \cite{YGu}, one sees that
the logarithmic term plays a positive role for problem \eqref{P1} to admit weak solutions when $N\geq8$
and this leads to some interesting phenomena.

This paper is organized as follows. In Section 2 we give some notations, definitions and
introduce some necessary lemmas. The main results of this paper are also stated in this section.
In Section 3 we prove some technical lemmas and the detailed proof of the main results
will be given in Section 4.

\par
\section{Preliminaries and the main results.}
\setcounter{equation}{0}

In this section, we first introduce some notations and definitions that will be used throughout the paper.
In what follows, we denote by $\|\cdot\|_{p}$ the $L^p(\Omega)$ norm for $1\leq p\leq \infty$.
The Sobolev space $H_0^2(\Omega)$ will be equipped with the norm $\|u\|:=\|u\|_{H_0^2(\Omega)}=\|\Delta u\|_2$,
which is equivalent to the full one due to Poincar\'{e}'s inequality, and its dual space is denoted by $H^{-2}(\Omega)$.
For other linear normed space $X$, we denote its norm by $\|\cdot\|_{X}$. We always use $\rightarrow$ and $\rightharpoonup$
to denote the strong and weak convergence in each Banach space, respectively, and use $C$
to denote (possibly different) generic positive constants. We always denote by
$B_R(x_0)$ the ball with radius $R$ centered at $x_0$ in $\mathbb{R}^{N}$.
Sometimes $B_R(0)$ will be simply written as $B_R$ if no confusion arises.
We use $\omega_N$ to denote the area of the unit sphere in $\mathbb{R}^{N}$.
For each $t>0$, $O(t)$ denotes
the quantity satisfying $|\frac{O(t)}{t}|\leq C$, and $o(t)$ means that $|\frac{o(t)}{t}|\rightarrow0$ as $t\rightarrow0$.
Set
\allowdisplaybreaks
\begin{align}
\label{rou-max}\rho_{max}:=\sup\{R>0: there \ exists \ an \ x \in \Omega \ such\ that\  B_{R}(x)\subset\Omega\}
\end{align}
and $o_{n}(1)$ is an infinitesimal as $n\rightarrow\infty$.

We use $S$ to denote the best embedding constant 
from $H_0^2(\Omega)$ to $L^{2^{**}}(\Omega)$,
i.e.,
\begin{equation}\label{embedding constant-2}
S=\inf\limits_{u\in H_0^2(\Omega)\backslash\{0\}}\dfrac{\|u\|^2}{\|u\|_{2^{**}}^{2}}.
\end{equation}
 $\lambda_{1}(\Omega)>0$ denotes the first eigenvalue of eigenvalue problem
\begin{eqnarray}\label{first eigenvalue}
\begin{cases}
\Delta^2u=\lambda u, &x\in\Omega,\\
u=\dfrac{\partial u}{\partial \nu}=0, &x\in\partial\Omega.
\end{cases}
\end{eqnarray}

Finally, following the ideas of \cite{Deng}, we introduce the following sets
\allowdisplaybreaks
\begin{align*}
A:&=\{(\lambda, \mu)|\ \mu>0, \lambda\in \mathbb{R} \},\\
B:&=\{(\lambda, \mu)|\ \mu<0, \lambda\in [0, \lambda_1(\Omega)), \frac{2}{N}
 \Big(\frac{\lambda_{1}(\Omega)-\lambda}{\lambda_{1}(\Omega)}\Big)
 ^{\frac{N}{4}}S^{\frac{N}{4}}+\frac{\mu}{2}|\Omega|>0\},\\
C:&=\{(\lambda, \mu)|\ \mu<0, \lambda\in \mathbb{R}, \frac{2}{N}
 S^{\frac{N}{4}}+\frac{\mu}{2}e^{-\frac{\lambda}{\mu}}|\Omega|>0\}.
\end{align*}





In this paper, we consider weak solutions to problem \eqref{P1} in the following sense.
\begin{definition}\label{de2.1}$\mathrm{\bf{(Weak \ solution)}}$
A function $u\in  H_{0}^{2}(\Omega)$ is called a weak solution to problem \eqref{P1},
if for all $\phi\in H_{0}^{2}(\Omega)$, it holds that
\begin{equation*}\label{}
\int_{\Omega}\Delta u\Delta \phi\mathrm{d}x-\lambda\int_{\Omega}u\phi\mathrm{d}x
-\mu\int_{\Omega}u\phi\ln u^{2}\mathrm{d}x- \int_{\Omega}|u|^{2^{**}-2}u\phi\mathrm{d}x=0.
\end{equation*}
\end{definition}
The energy functional associated with problem \eqref{P1} is given by
\begin{equation*}\label{}
I(u)=\frac{1}{2}\| u\|^{2}-\frac{1}{2}\lambda\|u\|_{2}^{2}+\frac{1}{2}\mu\|u\|_{2}^{2}
-\frac{1}{2}\mu\int_{\Omega}u^{2}\ln u^{2}\mathrm{d}x-\frac{1}{2^{**}}\|u\|_{2^{**}}^{2^{**}},\qquad \forall \ u\in H_{0}^{2}(\Omega).
\end{equation*}
Obviously, the energy functional
$I(u)$ is a $C^1$ functional in $H_{0}^{2}(\Omega)$ (see  \cite{Deng,YGu}).
From $I(u)=I(|u|)$,
we may assume that $u\geq0$ in the sequel.

For each $u\in H_{0}^{2}(\Omega)\setminus\{0\}$,
consider  the fibering maps $\psi_{u}(t)$: $(0,+\infty)\rightarrow \mathbb{R}$ defined by
\allowdisplaybreaks
\begin{align}
\psi_{u}(t)=I(tu)&=\frac{1}{2}t^2\| u\|^{2}-\frac{1}{2}\lambda t^2\|u\|_{2}^{2}+\frac{1}{2}
\mu t^2\|u\|_{2}^{2}-\frac{1}{2}\mu \|u\|_{2}^{2}t^2\ln t^2\nonumber\\
\label{fibering map}&\ \ \ -\frac{1}{2}\mu t^2\int_{\Omega}u^{2}\ln u^{2}\mathrm{d}x-\frac{1}{2^{**}}t^{2^{**}}\|u\|_{2^{**}}^{2^{**}}.
\end{align}

In order to obtain the existence of weak solutions, we introduce the Nehari manifold
\allowdisplaybreaks
\begin{align*}
\mathcal{N}&=\{u\in H_{0}^{2}(\Omega)\backslash\{0\}: \langle I'(u),u\rangle=0 \} \\
&=\{u\in H_{0}^{2}(\Omega)\backslash\{0\}: \psi_{u}'(1)=0 \}.
\end{align*}
Here we use $\langle\ ,\ \rangle$ to denote the pairing between $H^{-2}(\Omega)$ and $H_{0}^{2}(\Omega)$.
We usually split $\mathcal{N}$ into three disjoint parts as $\mathcal{N}=\mathcal{N}^{+}\cup \mathcal{N}^{0}\cup \mathcal{N}^{-}$, where
\allowdisplaybreaks  \begin{align*}
\mathcal{N}^{+}&=\{u\in N: \psi_{u}''(1)>0\},\\
\mathcal{N}^{-}&=\{u\in N: \psi_{u}''(1)<0\},\\
\mathcal{N}^{0}&=\{u\in N: \psi_{u}''(1)=0\}.
\end{align*}
It is well known that if $u\in H_0^2(\Omega)$ is a nontrivial solution to problem \eqref{P1}, then $u\in \mathcal{N}$.

Next we introduce the definition of a compactness condition, usually referred to as the $(PS)_c$ condition.

\begin{definition}\label{sequence}($(PS)_c$ condition)
Assume that $X$ is a real Banach space, $I:X\rightarrow \mathbb{R}$ is a $C^1$ functional and $c\in \mathbb{R}$.
We say that $\{u_{n}\}\subset H_0^2(\Omega)$ is a $(PS)_c$ sequence for $I$ if 
\begin{eqnarray*}
I(u_n)\rightarrow c\  \ and \ I'(u_{n})\rightarrow 0\ in  \ X^{-1}(\Omega)\ as \ n\rightarrow\infty,
\end{eqnarray*}
where $X^{-1}$ is the dual space of $X$.
Furthermore, we say that $I$ satisfies the $(PS)_c$ condition if any $(PS)_c$ sequence has a convergent subsequence.
\end{definition}

The following two lemmas will play a fundamental role in proving the existence of weak solutions when $\mu>0$.
The first one is the famous Mountain Pass Lemma, and the second one is Br\'{e}zis-Lieb's lemma.

\begin{lemma}\label{lem-Mountain Pass Lemma}(Mountain Pass Lemma \cite{MWillem})
Assume that $(X, \|\cdot\|_X)$ is a real Banach space, $I:X\rightarrow \mathbb{R}$ is a $C^1$ functional
and there exist $\beta>0$ and $r>0$ such that $I$ satisfies the following  mountain pass geometry:

(i) $I(u)\geq \beta>0$ if $\|u\|_X=r$;

(ii) there exists a $\overline{u}\in X$ such that $\|\overline{u}\|_X>r$ and $I(\overline{u})<0$.

Then there exist a $(PS)_{c_0}$ sequence $\{u_n\}\subset X$, i.e.,  $I(u_n)\rightarrow c_0$
and $I'(u_{n})\rightarrow 0$ in $X^{-1}$ as $n\rightarrow\infty$,
where $X^{-1}$ is the dual space of $X$ and
\begin{eqnarray*}
c_0:=\inf_{\gamma\in\Gamma}\max_{t\in[0,1]}I(\gamma(t))\geq \beta,
 \ \Gamma=\left\{\gamma\in C([0,1],X): \gamma(0)=0, I(\gamma(1))<0\right\},
\end{eqnarray*}
which is called the mountain pass level. Furthermore, $c_0$ is a critical vale of $I$ if $I$ satisfies the $(PS)_{c_0}$ condition.
\end{lemma}

\begin{lemma}\label{lem-Lieb}(Br\'{e}zis-Lieb's lemma \cite{HBr})
Let $p\in(1,\infty)$. Suppose that $\{u_n\}$
is a bounded sequence in $L^p(\Omega)$ and $u_n\rightarrow u$ a.e. in  $\Omega$.  Then
\begin{eqnarray*}
&\lim\limits_{n\rightarrow\infty}(\|u_n\|_p^p-\|u_n-u\|_p^p)=\|u\|_p^p.
\end{eqnarray*}
\end{lemma}

To deal with the logarithmic nonlinearity $u\ln u^{2}$, we need the following two lemmas.
Lemma \ref{lem-logarithmic inequality} can be directly verified and Lemma \ref{lem-logarithmic Sobolev inequality}
is the modified logarithmic Sobolev inequality.

\begin{lemma}\label{lem-logarithmic inequality}
$(1)$ For all $t\in(0,1]$, there holds that
\begin{equation}\label{logarithmic-1}
|t\ln t |\leq\frac{1}{e}.
\end{equation}

$(2)$  For any $\alpha,  \delta>0$, there exists a positive constant $C_{\alpha,\delta}$ such that
 \begin{equation}\label{logarithmic-2}
|\ln t|\leq C_{\alpha,\delta}(t^{\alpha}+t^{-\delta}), \qquad \forall\ t>0.
\end{equation}

$(3)$ For any $\sigma>0$, there holds that
 \begin{equation}\label{logarithmic-3}
\frac{\ln t}{t^\sigma}\leq \frac{1}{\sigma e}, \qquad \forall\ t>0.
\end{equation}
\end{lemma}

\begin{lemma}(Logarithmic Sobolev inequality)\label{lem-logarithmic Sobolev inequality}
Assume that $\Omega\subset \mathbb{R}^{N}(N\geq5)$ is a bounded domain
with smooth boundary. Then for any $u\in H_{0}^2(\Omega)$ and $a>0$, we have
\begin{eqnarray}\label{1-3}
\int_{\Omega}u^{2}\ln u^{2}\mathrm{d}x
\leq\frac{a^2}{\pi \widetilde{\lambda}_1}\| u\|^{2}
+\left(\ln \| u\|^{2}_{2}-N(1+\ln a)\right)\| u\|^{2}_{2},
\end{eqnarray}
where $\widetilde{\lambda}_1>0$ denotes the first eigenvalue of $-\Delta$ with zero Dirichlet boundary condition on $\Omega$.
\end{lemma}

\begin{proof}
According to \cite{LGross}, for any $u\in H^2(\mathbb{R}^{N})$ and $a>0$, one has
\begin{equation}\label{1-1}
\int_{\mathbb{R}^{N}}u^{2}\ln u^{2}\mathrm{d}x
\leq\frac{a^2}{\pi}\int_{\mathbb{R}^{N}}|\nabla u|^{2}\mathrm{d}x
+\left(\ln \| u\|^{2}_{L^2(\mathbb{R}^{N})}-N(1+\ln a)\right)\| u\|^{2}_{L^2(\mathbb{R}^{N})}.
\end{equation}
For any $u\in H_{0}^2(\Omega)$, extend the domain of $u$ to $\mathbb{R}^{N}$ by letting $u(x)=0$ for $x\in \mathbb{R}^{N}\setminus\Omega$ and still denote it by $u$. Then $u\in H^2(\mathbb{R}^{N})$ and therefore
the above inequality holds with $\mathbb{R}^{N}$ replaced by $\Omega$. Moreover,
by integrating by parts and applying Cauchy's inequality with $\varepsilon$, one has
\allowdisplaybreaks
\begin{align*}
\int_{\Omega}|\nabla u|^{2}\mathrm{d}x
&=\int_{\partial\Omega} u\dfrac{\partial u}{\partial \nu}\mathrm{d}s-\int_{\Omega} u\Delta u\mathrm{d}x=-\int_{\Omega} u\Delta u\mathrm{d}x\nonumber\\
&\leq\frac{\varepsilon}{2}\int_{\Omega} |u|^2\mathrm{d}x
+\frac{1}{2\varepsilon}\int_{\Omega} |\Delta u|^2\mathrm{d}x\nonumber\\
\label{1-1}&\leq\frac{\varepsilon}{2\widetilde{\lambda}_1}\int_{\Omega} |\nabla u|^2\mathrm{d}x
+\frac{1}{2\varepsilon}\int_{\Omega} |\Delta u|^2\mathrm{d}x.
\end{align*}
Taking $\varepsilon=\widetilde{\lambda}_1$ in the above inequality yields
\begin{eqnarray}\label{1-2}
\int_{\Omega} |\nabla u|^2\mathrm{d}x\leq\frac{1}{\widetilde{\lambda}_1}\int_{\Omega} |\Delta u|^2\mathrm{d}x.
\end{eqnarray}
Combining \eqref{1-1} with \eqref{1-2} one obtains \eqref{1-3}. This completes the proof.
 \end{proof}

The main existence results in the paper can be summarized into the following theorem.
\begin{theorem}\label{th}
$(i)$  Assume that $N\geq8$. If $(\lambda,\mu)\in A$, then problem \eqref{P1} has a nonnegative mountain
pass type solution, which is also a ground state solution.

$(ii)$  Assume that $N=8$. If $(\lambda,\mu)\in B\cup C$ and
$\dfrac{25\cdot1920e^{\frac{\lambda}{\mu}+\frac{34}{3}}}{\rho_{max}^{4}}<1$,
then there exists at least one nonnegative nontrivial weak solution to problem \eqref{P1},
where
$$\rho_{max}:=\sup\{R>0: there \ exists\ an\ x \in \ \Omega \ such \ that \ B_{R}(x)\subset\Omega\}.$$

$(iii)$  Assume that $N=5,6,7$. If $(\lambda,\mu)\in B\cup C$, then there exists
at least one nonnegative nontrivial weak solution to problem \eqref{P1}.
\end{theorem}

\begin{remark}
Notice that the assumption $\dfrac{25\cdot1920e^{\frac{\lambda}{\mu}+\frac{34}{3}}}{
\rho_{max}^{4}}<1$ implies that the domain $\Omega$ should be appropriately large.
\end{remark}

As for the nonexistence result for problem \eqref{P1}, we have the following theorem.

\begin{theorem}\label{th-2}
Assume that $N\geq5$. If  $\frac{-\mu(N-4)}{4}+\frac{\mu(N-4)}{4}\ln \big(\frac{-\mu(N-4)}{4}\big)+\lambda-\lambda_{1}(\Omega)\geq0$ and  $\mu<0$, then problem \eqref{P1} has no positive solutions.
\end{theorem}

\par
\section{Some technical lemmas}
\setcounter{equation}{0}

We begin this section with a lemma that gives some properties of the Nehari manifold and fibering maps,
which will help us to show that the mountain pass solution is a ground state solution to problem \eqref{P1} when $\mu>0$.

\begin{lemma}\label{lem-N}
Assume that $\mu>0$. Then, $\mathcal{N}=\mathcal{N}^-$. Moreover,
for any $u\in H_0^2(\Omega)\backslash\{0\}$ there exists
a unique $t(u)>0$ such that $t(u)u\in \mathcal{N}$.
\end{lemma}

\begin{proof}
For each $u\in \mathcal{N}$, we have
\begin{equation*}
\psi_u'(1)=
\| u\|^{2}-\lambda\|u\|_{2}^{2}
-\mu\int_{\Omega}u^{2}\ln u^{2}\mathrm{d}x-\|u\|_{2^{**}}^{2^{**}}=0.
\end{equation*}
Consequently, it follows from $\mu>0$ that
\allowdisplaybreaks  \begin{align*}
\psi''_u(1)&=
\| u\|^{2}-\lambda\|u\|_{2}^{2}-2\mu\|u\|_{2}^{2}
-\mu\int_{\Omega}u^{2}\ln u^{2}\mathrm{d}x-(2^{**}-1)\|u\|_{2^{**}}^{2^{**}}\\
&=
-2\mu\|u\|_{2}^{2}
-(2^{**}-2)\|u\|_{2^{**}}^{2^{**}}\\
&<0,
 \end{align*}
which implies that $u\in \mathcal{N}^-(\subset \mathcal{N})$.

For any $u\in H_0^2(\Omega)\backslash\{0\}$, according the definition of $\psi_{u}(t)$ in \eqref{fibering map},
we deduce that
\begin{eqnarray}\label{-infty}
\lim\limits_{t\rightarrow0}\psi_{u}(t)=0,\quad \lim\limits_{t\rightarrow+\infty}\psi_{u}(t)=-\infty,
\end{eqnarray}
and $\psi_{u}(t)$ is positive for $t>0$ suitably small (since $\mu>0$), which imply that there exists a unique $t(u)>0$ such that
$\max\limits_{t\geq0}\psi_u(t)=\psi_u(t(u))$. Indeed, the existence of $t(u)$ is obvious. If there were $0<t_1(u)<t_2(u)$ such
that $\max\limits_{t\geq0}\psi_u(t)=\psi_u(t_1(u))=\psi_u(t_2(u))$, then there must be a $t_3(u)\in(t_1(u),t_2(u))$ which is
a local minimum of $\psi_u(t)$. Direct computation shows that $t_3(u)u\in \mathcal{N}^+$, which is impossible since $\mathcal{N}=\mathcal{N}^-$.
This  completes the proof.
\end{proof}

\begin{remark}\label{Nehari manifold}
(i) Following  from Lemma \ref{lem-N}, if $\mu>0$, one can give an equivalent characterization of $c_0$
defined in Lemma \ref{lem-Mountain Pass Lemma}, i.e., $c_0=c_{\mathcal{N}}$, where
\begin{eqnarray*}
c_\mathcal{N}:=\inf\limits_{u\in \mathcal{N}}I(u)=\inf\limits_{u\in H_0^2(\Omega)\backslash\{0\}}\max\limits_{t\geq0}I(tu).
\end{eqnarray*}
The proof of this equivalent characterization can be founded in Theorem
$4.2$ of \cite{MWillem}. 

(ii) Since all the nontrivial critical points of $I$ belong to $\mathcal{N}$,
if $u$ is a nontrivial critical point of $I$ satisfying $I(u)=c_\mathcal{N}$, it must be a
ground state solution to problem \eqref{P1}.
\end{remark}

In order to apply the Mountain Pass Lemma, we first verify the mountain pass  geometry for $I$
 when $(\lambda,\mu)\in A\cup B\cup C$.

\begin{lemma}\label{lem-mountain pass geometry structure} Assume that $N\geq 5$ and $(\lambda,\mu)\in A\cup B\cup C$. Then the functional $I(u)$ satisfies the mountain pass geometry.
\end{lemma}

\begin{proof}
The proof is divided into three cases.

\noindent{\bf Case 1:} $(\lambda,\mu)\in A$.

Since $\mu>0$, by using Sobolev embedding inequality and applying \eqref{logarithmic-3} with
$\sigma=\frac{2^{**}-2}{2}$, one has
\allowdisplaybreaks  \begin{align*}
&\ \ \ \ -\frac{1}{2}\lambda\|u\|_{2}^{2}+\frac{1}{2}\mu\|u\|_{2}^{2}
-\frac{1}{2}\mu\int_{\Omega}u^{2}\ln u^{2}\mathrm{d}x
\geq-\frac{1}{2}\lambda\|u\|_{2}^{2}
-\frac{1}{2}\mu\int_{\Omega}u^{2}\ln u^{2}\mathrm{d}x\\
&=-\frac{\mu}{2}\int_{\Omega}u^{2}(\frac{\lambda}{\mu}
+ \ln u^{2})\mathrm{d}x
=-\frac{\mu}{2}e^{-\frac{\lambda}{\mu}}\int_{\Omega}e^{\frac{\lambda}{\mu}}u^{2}
 \ln (e^{\frac{\lambda}{\mu}}u^{2})\mathrm{d}x\\
 &=-\frac{\mu}{2}e^{-\frac{\lambda}{\mu}}\int_{\{e^{\frac{\lambda}{\mu}}u^{2}\geq1\}}
 e^{\frac{\lambda}{\mu}}u^{2}
 \ln (e^{\frac{\lambda}{\mu}}u^{2})\mathrm{d}x
-\frac{\mu}{2}e^{-\frac{\lambda}{\mu}}
\int_{\{e^{\frac{\lambda}{\mu}}u^{2}\leq1\}}
e^{\frac{\lambda}{\mu}}u^{2}
 \ln (e^{\frac{\lambda}{\mu}}u^{2})\mathrm{d}x\\
 &\geq
 -\frac{\mu}{2}e^{-\frac{\lambda}{\mu}}
 \int_{\{e^{\frac{\lambda}{\mu}}u^{2}\geq1\}}
 e^{\frac{\lambda}{\mu}}u^{2}
 \ln (e^{\frac{\lambda}{\mu}}u^{2})\mathrm{d}x
 \geq
 -\frac{\mu}{2}e^{-\frac{\lambda}{\mu}-1}\frac{2}{2^{**}-2}
 \int_{\{e^{\frac{\lambda}{\mu}}u^{2}\geq1\}}
 (e^{\frac{\lambda}{\mu}}u^{2}) ^{\frac{2^{**}}{2}}\mathrm{d}x\\
 &\geq
 -\frac{\mu}{2^{**}-2}e^{-\frac{\lambda}{\mu}-1+\frac{2^{**}\lambda}{2\mu}}
S^{-\frac{2^{**}}{2}}\|u\|^{2^{**}},
\end{align*}
which ensures that
\allowdisplaybreaks  \begin{align*}
I(u)\geq\frac{1}{2}\| u\|^{2}-\frac{S^{-\frac{2^{**}}{2}}}{2^{**}}\|u\|^{2^{**}}
 -\frac{\mu}{2^{**}-2}e^{-\frac{\lambda}{\mu}-1+\frac{2^{**}\lambda}{2\mu}}
S^{-\frac{2^{**}}{2}}\|u\|^{2^{**}}.
 \end{align*}
Therefore, there exist positive constants $\beta$ and $r$ such that
\begin{eqnarray*}\label{}
 I(u)\geq\beta\ for\ all\ \|u\|=r.
\end{eqnarray*}

On the other hand, for any $u\in H_0^2(\Omega)\backslash\{0\}$,
in view of \eqref{-infty}, there exists a $t_u>0$ suitably large such that $\|t_uu\|>r$ and $\psi_{u}(t_u)=I(t_uu)<0$.

\noindent{\bf Case 2:} $(\lambda,\mu)\in B$.

Since $\mu<0$, recalling \eqref{logarithmic-1}, one has
\allowdisplaybreaks  \begin{align*}
&\ \ \ \ \frac{1}{2}\mu\|u\|_{2}^{2}
-\frac{1}{2}\mu\int_{\Omega}u^{2}\ln u^{2}\mathrm{d}x
=
-\frac{\mu}{2}\int_{\Omega}u^{2}\ln (e^{-1}u^{2})\mathrm{d}x\\
 &=-\frac{\mu}{2}e\int_{\{e^{-1}u^{2}\geq1\}}
 e^{-1}u^{2}
 \ln (e^{\frac{\lambda}{\mu}}u^{2})\mathrm{d}x
-\frac{\mu}{2}e
\int_{\{e^{-1}u^{2}\leq1\}}
e^{-1}u^{2}
 \ln (e^{\frac{\lambda}{\mu}}u^{2})\mathrm{d}x\\
 &\geq-\frac{\mu}{2}e
\int_{\{e^{-1}u^{2}\leq1\}}
e^{-1}u^{2}
 \ln (e^{\frac{\lambda}{\mu}}u^{2})\mathrm{d}x\geq\frac{\mu}{2}|\Omega|,
\end{align*}
which, together with $\eqref{first eigenvalue}$ and Sobolev embedding
inequality, yields
 \allowdisplaybreaks  \begin{align*}
I(u)\geq\frac{1}{2}\left(\frac{\lambda_{1}(\Omega)-\lambda}{\lambda_{1}(\Omega)}\right)\| u\|^{2}-\frac{S^{-\frac{2^{**}}{2}}}{2^{**}}\|u\|^{2^{**}}
+\frac{\mu}{2}|\Omega|.
\end{align*}
Set
\allowdisplaybreaks  \begin{align*}
g(t)=\frac{1}{2}\left(\frac{\lambda_{1}(\Omega)-\lambda}{\lambda_{1}(\Omega)}\right)t^{2}
-\frac{S^{-\frac{2^{**}}{2}}}{2^{**}}t^{2^{**}}
+\frac{\mu}{2}|\Omega|,\qquad t>0.
 \end{align*}
By a direct calculation,   $g$ takes its maximum at $t_0:=\left(\frac{\lambda_{1}(\Omega)-\lambda}{\lambda_{1}(\Omega)}\right)
^{\frac{N-4}{8}}S^{\frac{N}{8}}$
 and
 \begin{eqnarray*}\label{}
0<g(t_0)=\frac{2}{N}
 \left(\frac{\lambda_{1}(\Omega)-\lambda}{\lambda_{1}(\Omega)}\right)
 ^{\frac{N}{4}}S^{\frac{N}{4}}+\frac{\mu}{2}|\Omega|,
\end{eqnarray*}
due to  $(\lambda,\mu)\in B$,
which implies that  there exist positive constants $\beta:=g(t_0)$ and $r:= t_0$ such that
\begin{eqnarray*}\label{}
 I(u)\geq\beta\ for\ all\ \|u\|=r.
\end{eqnarray*}
Applying a similar argument to Case 1, one can show that there is a $v\in H_0^2(\Omega)$ such that $\|v\|>r$ and $I(v)<0$.

\noindent{\bf Case 3:} $(\lambda,\mu)\in C$.

Since $\mu<0$, using \eqref{logarithmic-1} again,
one has
\allowdisplaybreaks  \begin{align*}
&\ \ \ \  -\frac{1}{2}\lambda\|u\|_{2}^{2}+\frac{1}{2}\mu\|u\|_{2}^{2}
-\frac{1}{2}\mu\int_{\Omega}u^{2}\ln u^{2}\mathrm{d}x
=-\frac{\mu}{2}\int_{\Omega}u^{2}(\frac{\lambda}{\mu}
- 1+\ln u^{2})\mathrm{d}x\\
&=-\frac{\mu}{2}e^{-\frac{\lambda}{\mu}+1}\int_{\Omega}
 e^{\frac{\lambda}{\mu}-1}u^{2}
 \ln (e^{\frac{\lambda}{\mu}-1}u^{2})\mathrm{d}x\\
 &=-\frac{\mu}{2}e^{-\frac{\lambda}{\mu}+1}
 \int_{\{e^{\frac{\lambda}{\mu}-1}u^{2}\geq1\}}
 e^{\frac{\lambda}{\mu}-1}u^{2}
 \ln (e^{\frac{\lambda}{\mu}-1}u^{2})\mathrm{d}x
-\frac{\mu}{2}e^{-\frac{\lambda}{\mu}+1}
\int_{\{e^{\frac{\lambda}{\mu}-1}u^{2}\leq1\}}
 e^{\frac{\lambda}{\mu}-1}u^{2}
 \ln (e^{\frac{\lambda}{\mu}-1}u^{2})\mathrm{d}x\\
 &\geq-\frac{\mu}{2}e^{-\frac{\lambda}{\mu}+1}
\int_{\{e^{\frac{\lambda}{\mu}-1}u^{2}\leq1\}}
 e^{\frac{\lambda}{\mu}-1}u^{2}
 \ln (e^{\frac{\lambda}{\mu}-1}u^{2})\mathrm{d}x\geq\frac{\mu}{2}
 e^{-\frac{\lambda}{\mu}}|\Omega|,
 \end{align*}
and
\begin{eqnarray*}\label{}
\begin{split}
I(u)\geq\frac{1}{2}\| u\|^{2}-\frac{S^{-\frac{2^{**}}{2}}}{2^{**}}\|u\|^{2^{**}}
+\frac{\mu}{2}e^{-\frac{\lambda}{\mu}}|\Omega|.\\
\end{split}
\end{eqnarray*}
Set
\begin{eqnarray*}\label{}
\widetilde{g}(t):=\frac{1}{2}t^{2}-\frac{S^{-\frac{2^{**}}{2}}}{2^{**}}t^{2^{**}}
+\frac{\mu}{2}e^{-\frac{\lambda}{\mu}}|\Omega|,\qquad t>0.
\end{eqnarray*}
A direct calculation shows that    $\widetilde{g}$ takes its maximum at $\widetilde{t}_0:=S^{\frac{N}{8}}$
 and
 \begin{eqnarray*}\label{}
0<\widetilde{g}(\widetilde{t}_0)=\frac{2}{N}
 S^{\frac{N}{4}}+\frac{\mu}{2}e^{-\frac{\lambda}{\mu}}|\Omega|,
\end{eqnarray*}
 due to  $(\lambda,\mu)\in C$,
which implies that  there exist positive constants $\beta:=\widetilde{g}(\widetilde{t}_0)$ and $r:=\widetilde{t}_0$ such that
\begin{eqnarray*}\label{}
 I(u)\geq\beta\ for\ all\ \|u\|=r.
\end{eqnarray*}
Applying a similar argument to Case 1, one can show that there is a $v\in H_0^2(\Omega)$ such that $\|v\|>r$ and $I(v)<0$.

In a conclusion, $I$ satisfies the mountain pass geometry  around $0$ if $(\lambda,\mu)\in A \cup B\cup C$. The proof is complete.
\end{proof}

Now we show the boundedness of $(PS)_c$ sequence of $I$.

\begin{lemma}\label{lem-bound}
Assume that $N \geq5$,  $\mu \in \mathbb{R}\backslash\{0\}$ and $\lambda\in R$.
If $\{u_{n}\}$  is a $(PS)_{c}$ sequence of $I$,
then $\{u_n\}$ must be bounded in $H^2_0(\Omega)$ for all $c\in \mathbb{R}$.
\end{lemma}

\begin{proof}
Recalling the definition of the $(PS)_c$ sequence $\{u_n\}$, we have
\begin{eqnarray*}
I(u_n)\rightarrow c\  \ and \ I'(u_{n})\rightarrow 0\ in  \ H_0^2(\Omega)\ as \ n\rightarrow\infty.
\end{eqnarray*}

When $\mu<0$, from the above equalities and using \eqref{logarithmic-1}, we have,
as $n\rightarrow\infty$, that
\allowdisplaybreaks
\begin{align*}\label{}
 &\ \ \ \ c+1+o_n(1)\|u_n\|
\geq
I(u_n)-\frac{1}{2^{**}}\langle I'(u_n),u_n\rangle\nonumber\\
&=
\frac{2}{N}\| u_n\|^{2}-\frac{2}{N}\lambda\|u_n\|_{2}^{2} +\frac{1}{2}\mu\|u_n\|_{2}^{2}
-\frac{2}{N}\mu\int_{\Omega}u_n^{2}\ln u_n^{2}\mathrm{d}x\nonumber\\
&=
\frac{2}{N}\| u_n\|^{2}
-\frac{2}{N}\mu\int_{\Omega}u_n^{2}\ln \big(e^{\frac{\lambda}{\mu}-\frac{N}{4}}u_n^{2}\big)\mathrm{d}x\nonumber\\
&\geq
\frac{2}{N}\| u_n\|^{2}
-\frac{2}{N}\mu e^{-\frac{\lambda}{\mu}+\frac{N}{4}}
\int_{\{e^{\frac{\lambda}{\mu}-\frac{N}{4}}u_n^{2}\leq1\}}
e^{\frac{\lambda}{\mu}-\frac{N}{4}}u_n^{2}\ln \big(e^{\frac{\lambda}{\mu}-\frac{N}{4}}u_n^{2}\big)\mathrm{d}x\nonumber\\
&\geq
\frac{2}{N}\| u_n\|^{2}
+\frac{2}{N}\mu e^{-\frac{\lambda}{\mu}+\frac{N}{4}-1}|\Omega|,\nonumber
\end{align*}
which shows that $\{u_n\}$  is bounded in $H_0^2(\Omega)$.

Assume now $\mu>0$. It follows from the definition of  $\{u_n\}$, we have, as $n\rightarrow\infty$, that
\begin{eqnarray*}\label{}
 c+1+o_n(1)\|u_n\|
\geq I(u_n)-\frac{1}{2}\langle I'(u_n), u_n\rangle=\frac{1}{2}\mu\|u_n\|_{2}^{2}
+\frac{2}{N}\|u_n\|_{2^{**}}^{2^{**}}
\geq\frac{1}{2}\mu\|u_n\|_{2}^{2},
\end{eqnarray*}
which shows that
\begin{eqnarray*}
\|u_n\|_{2}^{2}\leq C+C\|u_n\|.
\end{eqnarray*}
From  this   and  \eqref{1-3},
we can deduce, for $n$ large enough, that
\begin{eqnarray}\label{logarithmic Sobolev inequality}
\int_{\Omega}u_n^{2}\ln u_n^{2}\mathrm{d}x
&\leq
\frac{a^2}{\pi \widetilde{\lambda}_{1}}\| u_{n}\|^2
+\big[\ln (C+C\|u_n\|)-N(1+\ln a)\big](C+C\|u_n\|).
\end{eqnarray}
On the other hand, one has, as $n\rightarrow\infty$, that
\allowdisplaybreaks  \begin{align*}
 &\ \ \ \ c+1+o_n(1)\|u_n\|
\geq
I(u_n)-\frac{1}{2^{**}}\langle I'(u_n),u_n\rangle\\
&=
\frac{2}{N}\| u_n\|^{2}-\frac{2}{N}\lambda\|u_n\|_{2}^{2} +\frac{1}{2}\mu\|u_n\|_{2}^{2}
-\frac{2}{N}\mu\int_{\Omega}u_n^{2}\ln u_n^{2}\mathrm{d}x\\
&\geq
\frac{2}{N}\| u_n\|^{2}-C\|u_n\|_{2}^{2}
-\frac{2}{N}\mu\int_{\Omega}u_n^{2}\ln u_n^{2}\mathrm{d}x.
\end{align*}
Then, in view of  \eqref{logarithmic Sobolev inequality} and $\mu>0$, one has, as $n\rightarrow\infty$, that
 \allowdisplaybreaks  \begin{align*}
 &\ \ \ \  c+1+o_n(1)\|u_n\|\\
&\geq
\frac{2}{N}\| u_n\|^{2}-(C+C\|u_n\|)-\frac{2}{N}\mu\Big(\frac{a^2}{\pi\widetilde{\lambda}_{1}}\| u_{n}\|^2
+\big[\ln (C+C\|u_n\|)-N(1+\ln a)\big](C+C\|u_n\|)\Big)\\
&=
\frac{2}{N}\Big(1-\frac{\mu a^2}{\pi\widetilde{\lambda}_{1}}\Big)\| u_n\|^{2}
-\frac{2}{N}\mu(C+C\|u_n\|)\ln (C+C\|u_n\|) +[2\mu N(1+\ln a)-1](C+C\|u_n\|)\\
&\geq
\frac{1}{N}\| u_n\|^{2}
-\frac{2}{e\sigma N}\mu (C+C\|u_n\|)^{1+\sigma}+[2\mu N(1+\ln a)-1](C+C\|u_n\|),
\end{align*}
where we choose $a>0$ with $\frac{\mu a^2}{\pi\widetilde{\lambda}_{1}}<\frac{1}{2}$,
and the last inequality is justified by \eqref{logarithmic-3} with $\sigma<1$.
Therefore, $\{u_n\}$ is a bounded sequence in $H_0^2(\Omega)$.
The proof is complete.
\end{proof}

On the basis of Lemma \ref{lem-bound}, we can prove that the $(PS)_c$ condition is valid for some $c$ when $\mu>0$.

\begin{lemma}\label{lem-PS-1}
Assume that $N\geq5$, $\mu>0$, and $\lambda\in \mathbb{R}$.
If $c<c(S)$, then $I$  satisfies the $(PS)_c$ condition, where $c(S):=\dfrac{2}{N}S^{\frac{N}{4}}$.
\end{lemma}

\begin{proof}
Let $\{u_{n}\}$ be a $(PS)_{c}$ sequence of $I$.
According to Lemma \ref{lem-bound}, $\{u_n\}$ is bounded in $H_0^2(\Omega)$.
Consequently,  by using the Sobolev embedding  one sees that there is a subsequence of $\{u_{n}\}$
(which we still denote by $\{u_{n}\}$) such that,  as $n\rightarrow\infty$,
\allowdisplaybreaks
\begin{eqnarray}\label{H02-convergence}
\ \ \ \ \ \ \ \ \ \ \ \ \begin{cases}
u_{n}\rightharpoonup u \ in \ H_{0}^2(\Omega),\\
u_{n}\rightarrow u \ in \ L^r(\Omega)\ (1\leq r<2^{**}),\\
u_{n}\rightarrow u \ in \ H_0^1(\Omega),\\
|u_{n}|^{2^{**}-2}u_n\rightharpoonup |u|^{2^{**}-2}u \ in \ L^{\frac{2^{**}}{2^{**}-1}}(\Omega),\\
u_{n}\rightarrow u \ \ a.e.\ in \ \Omega.
\end{cases}
\end{eqnarray}
Next, we are going to prove
\begin{equation}\label{logarithmi-convergence}
\lim_{n\rightarrow\infty}\int_{\Omega}u_n^{2}\ln u_n^2\mathrm{d}x
=\int_{\Omega}u^{2}\ln u^2\mathrm{d}x.
\end{equation}
Indeed,
since $u_n\rightarrow u$ a.e. in $\Omega$ as $n\rightarrow\infty$, we get
\begin{eqnarray}\label{dui-1}
u_n^2\ln u_n^2\rightarrow u^2\ln u^2\ a.e.\ in\ \Omega\ as \ n\rightarrow\infty.
\end{eqnarray}
Moreover,   combining  \eqref{logarithmic-1} and \eqref{logarithmic-3}   %
  we get
\begin{eqnarray}\label{dui-2}
|u_n^2\ln u_n^2|
\leq
\dfrac{1}{e}+\frac{1}{e\delta}u_n^{2(1+\delta)}
\rightarrow\dfrac{1}{eq}+\frac{1}{e\delta}u^{2(1+\delta)}\ in\ L^1(\Omega)\ as\ n\rightarrow\infty,
\end{eqnarray}
where $\delta>0$ is chosen to satisfy $2(1+\delta)<2^{**}$.
Then \eqref{logarithmi-convergence} follows from \eqref{dui-1}, \eqref{dui-2} and Lebesgue's dominated convergence theorem.

To show that $u_n\rightarrow n$ in $H_0^2(\Omega)$ as $n\rightarrow\infty$,
set $w_n=u_n-u$. Then $\{w_n\}$ is also a bounded sequence in $H_0^2(\Omega)$.
So there exists a subsequence of $\{w_{n}\}$ (which we still denoted by $\{w_{n}\}$) such that
\begin{equation}\label{limit-wn-l}
\lim\limits_{n\rightarrow\infty}\|w_n\|^2=l\geq0.
\end{equation}
We claim that $l=0$. Indeed, by virtue of the weak convergence $u_{n}\rightharpoonup u$ in $H_{0}^2(\Omega)$, we have
\begin{eqnarray}\label{H02-Brezis}
\|u_n\|^{2}
=\| w_n\|^2+\| u\|^{2}+o_n(1), \qquad  n\rightarrow\infty.
\end{eqnarray}
From the boundedness of $\{ u_n\}$ and the Sobolev embedding, we can easily conclude that $\| u_n\|_{2^{**}}\leq C$.
This, together with $u_{n}\rightarrow u$ a.e. in $\Omega$, implies that we can apply Br\'{e}zis-Lieb's lemma to obtain
\begin{equation}\label{2**-Brezis}
\|u_n\|_{2^{**}}^{2^{**}}=\|w_n\|_{2^{**}}^{2^{**}}+
\|u\|_{2^{**}}^{2^{**}}+o_n(1),\qquad n\rightarrow\infty.
\end{equation}
Therefore, in accordance with \eqref{H02-convergence}, \eqref{logarithmi-convergence},  \eqref {H02-Brezis},  \eqref{2**-Brezis}  and the assumption that $I'(u_{n})\rightarrow 0$ in $H^{-2}(\Omega)$ as $n\rightarrow\infty$, we have
\begin{eqnarray*}\label{}
o_n(1)=\langle I'(u_n),u_n\rangle
=\langle I'(u),u\rangle+\|w_n\|^2-\|w_n\|_{2^{**}}^{2^{**}}+o_n(1),\qquad n\rightarrow\infty,
\end{eqnarray*}
and for all $\phi\in H_0^2(\Omega)$,
\begin{eqnarray*}\label{}
o_n(1)=\langle I'(u_n),\phi\rangle
=\langle I'(u),\phi\rangle+o_n(1), \qquad n\rightarrow\infty,
\end{eqnarray*}
which implies that $u$ is a weak solution to problem \eqref{P1}.
Combining the above two equalities, and taking $\phi=u$, one obtains
\begin{equation}\label{I'=0}
\langle I'(u),u\rangle=0,
\end{equation}
and
\begin{equation}\label{S-1}
\|w_n\|^2-\|w_n\|_{2^{**}}^{2^{**}}=o_n(1),\qquad  n\rightarrow\infty.
\end{equation}
 In addition, by  the Sobolev embedding one has
 \begin{equation}\label{S-2}
  \|w_n\|_{2^{**}}^{2^{**}}\leq S^{-\frac{2^{**}}{2}}\| w_n\|^{2^{**}}<C,\qquad \forall\ n\in \mathbb{N}.
\end{equation}
In view of  \eqref{H02-convergence}, \eqref{S-1} and \eqref{S-2},  it is easy to see that  there is a subsequence of $\{w_{n}\}$ such that
\begin{equation}\label{limit-wn-2}
\lim\limits_{n\rightarrow\infty}\|w_n\|_{2^{**}}^{2^{**}}
=\lim\limits_{n\rightarrow\infty}\|w_n\|^2=l\geq0.
\end{equation}
Letting $n\rightarrow\infty$ in \eqref{S-2}, we have $l\leq S^{-\frac{2^{**}}{2}}l^{\frac{2^{**}}{2}}$.
If $l>0$, then
\begin{equation}\label{S-3}
l\geq S^{\frac{2^{**}}{2^{**}-2}}=S^{\frac{N}{4}}.
\end{equation}

On one hand, with the help of  \eqref{I'=0}, we can derive
\begin{eqnarray*}\label{}
I(u)
=I(u)-\frac{1}{2}\langle I'(u),u\rangle
=\frac{1}{2}\mu\|u\|_{2}^{2}
+\frac{2}{N}\|u\|_{2^{**}}^{2^{**}}\geq0,
\end{eqnarray*}
since $\mu>0$.

On the other hand, in view of \eqref{H02-convergence}, \eqref{logarithmi-convergence}, \eqref{H02-Brezis}, \eqref{2**-Brezis} and the fact that $I(u_n)=c+o_n(1)$ as $n\rightarrow\infty$,
we have
\begin{eqnarray*}
o_n(1)+c=I(u_n)
=I(u)+\frac{1}{2}\| w_n\|^2-\frac{1}{2^{**}}\|w_n\|_{2^{**}}^{2^{**}}+o_n(1),\qquad  n\rightarrow\infty.
\end{eqnarray*}
Thus, we   can deduce that
$$I(u)=c-\dfrac{1}{2}\| w_n\|^2+\dfrac{1}{2^{**}}\|w_n\|_{2^{**}}^{2^{**}}+o_n(1), \qquad n\rightarrow\infty.$$
From this and recalling \eqref{limit-wn-2} and \eqref{S-3} one has
\begin{eqnarray*}
I(u)=c-\Big(\frac{1}{2}-\frac{1}{2^{**}}\Big)l\leq c-\frac{2}{N}S^{\frac{N}{4}}<0,
\end{eqnarray*}
a contradiction. Thus, $\lim\limits_{n\rightarrow\infty}\|w_n\|^2=l=0$,
i.e., $u_n\rightarrow u$ in $H_0^2(\Omega)$ as $n\rightarrow\infty$. The proof is complete.
\end{proof}

The case $\mu<0$ is quite different from the case $\mu>0$ since we can not check
the $(PS)_{c}$ condition in general. However, we can still obtain a
nontrivial weak solution by apply the Mountain Pass Lemma without $(PS)_c$ condition.

\begin{lemma}\label{lem-PS-2}
Assume that $N\geq5$,  $\mu<0$, $\lambda\in \mathbb{R}$ and $c\in (-\infty, 0)\cup (0,c(S))$, where $c(S)$ is given in Lemma \ref{lem-PS-1}.
If $\{u_{n}\}$  is a $(PS)_{c}$ sequence of $I$,  then there exists a $u\in H_0^2(\Omega)
\backslash\{0\}$  such that $u_n\rightharpoonup u$
in $H_0^2(\Omega)$ as $n\rightarrow\infty$ and $u$ is a
nontrivial weak solution to problem \eqref{P1}.
\end{lemma}

\begin{proof}
From Lemma \ref{lem-bound}, we know that $\{u_n\}$ is bounded in $H_0^2(\Omega)$.
Therefore, \eqref{H02-convergence} and \eqref{logarithmi-convergence} are valid.
As was done in proving Lemma \ref{lem-PS-1}, set $w_n=u_n-u$.
Then $\{w_n\}$ is bounded in $H_0^2(\Omega)$, and there exists a subsequence of $\{w_{n}\}$,
still denoted by $\{w_{n}\}$, such that
$$\lim\limits_{n\rightarrow\infty}\|w_n\|^2=l\geq0.$$
In view of \eqref{H02-convergence}
and Br\'{e}zis-Lieb's lemma, we know that \eqref{H02-Brezis} and \eqref{2**-Brezis} are still valid. From the definition of the $(PS)_c$ sequence, we have,  as $n\rightarrow\infty$,
\begin{eqnarray}\label{Ps-condition}
I'(u_{n})\rightarrow 0\ in \ H^{-2}(\Omega),\ \ \ \ \ I(u_n)\rightarrow c.
\end{eqnarray}
This, together with  \eqref{H02-convergence}, \eqref{H02-Brezis} and \eqref{2**-Brezis}, shows that,
for any $\phi\in H_0^2(\Omega)$,
\begin{eqnarray*}\label{}
o_n(1)=\langle I'(u_n),\phi\rangle
=\langle I'(u),\phi\rangle+o_n(1), \qquad n\rightarrow\infty,
\end{eqnarray*}
i.e., $u$ is a weak solution to problem \eqref{P1}.

Next, we claim that $u\neq0$. Assume on the contrary that $u=0$.
Following from \eqref{H02-convergence}, \eqref{logarithmi-convergence}, \eqref{H02-Brezis}, \eqref{2**-Brezis} and
\eqref{Ps-condition}, we have
\begin{eqnarray}\label{u=0-2}
o_n(1)=\langle I'(u_n),u_n\rangle
=\|w_n\|^2-\|w_n\|_{2^{**}}^{2^{**}}+o_n(1),\qquad n\rightarrow\infty,
\end{eqnarray}
and
\begin{eqnarray}\label{u=0-I}
o_n(1)+c=I(u_n)
=\frac{1}{2}\| w_n\|^2-\frac{1}{2^{**}}\|w_n\|_{2^{**}}^{2^{**}}+o_n(1),\qquad  n\rightarrow\infty.
\end{eqnarray}
 It then follows from 
 \eqref{u=0-2} that there is a subsequence of $\{w_{n}\}$ such that
\begin{equation*}\label{}
\lim\limits_{n\rightarrow\infty}\|w_n\|_{2^{**}}^{2^{**}}
=\lim\limits_{n\rightarrow\infty}\|w_n\|^2=l\geq0.
\end{equation*}
Letting $n\rightarrow\infty$ in \eqref{S-2}, we also have $l\leq S^{-\frac{2^{**}}{2}}l^{\frac{2^{**}}{2}}$.
In what follows, we are going to show that $l>0$. In fact, if $l=0$, one has
\begin{equation*}
\lim\limits_{n\rightarrow\infty}\|w_n\|_{2^{**}}^{2^{**}}
=\lim\limits_{n\rightarrow\infty}\|w_n\|^2=0.
\end{equation*}
This, together with \eqref{u=0-I}, implies that $I(u_n)\rightarrow 0$ as $n\rightarrow\infty$,
which contradicts with the assumption that $c\neq0$.
Therefore, $l>0$ and \eqref{S-3} still holds. Letting $n\rightarrow\infty$ in \eqref{u=0-I}
and recalling \eqref{S-3}, one has
\begin{eqnarray*}\label{}
c
=\Big(\frac{1}{2}-\frac{1}{2^{**}}\Big)l\geq\dfrac{2}{N}S^{\frac{N}{4}},
\end{eqnarray*}
a contradiction. Therefore, $u$ is a nontrivial weak solution to problem \eqref{P1}.
This completes the proof.
\end{proof}

The successful application of the Mountain Pass Lemma in proving the existence of weak solutions is based on the following lemma.

\begin{lemma}\label{crucial condition}
Assume that $N\geq5$ and $(\lambda, \mu) \in A\cup B\cup C$. If
 there exists a $u^*\in H_0^2(\Omega)\backslash\{0\}$ such that
\begin{eqnarray}\label{condition}
\sup_{t\geq0}I(tu^*)<c(S),
\end{eqnarray}
where $c(S)$ is given in Lemma \ref{lem-PS-1}, then problem \eqref{P1} possesses a nontrivial weak solution.
\end{lemma}

\begin{proof}
From Lemma \ref{lem-mountain pass geometry structure}, one sees that $I$ has a mountain pass geometry around $0$
if $(\lambda, \mu)\in A \cup B \cup C$. Thus, according to the Mountain Pass Lemma without $(PS)_c$
condition stated in Lemma \ref{lem-Mountain Pass Lemma}, there exists a sequence $\{u_n\}\subset H_0^2(\Omega)$
such that $I(u_{n})\rightarrow c_0$ and $I'(u_{n})\rightarrow 0$ in $H^{-2}(\Omega)$ as $n\rightarrow\infty$,
where
\begin{equation}\label{c0}
0<\beta\leq c_0=\inf_{\gamma\in\Gamma}\max_{t\in[0,1]}I(\gamma(t))\
and\ \
\Gamma=\{\gamma\in C([0,1],H_{0}^{2}(\Omega)): \gamma(0)=0, \gamma(1)=t_{u^*}u^*\},
\end{equation}
and $t_{u^*}>0$ is determined in the proof of Lemma \ref{lem-mountain pass geometry structure} such that $I(t_{u^*}u^*)<0$.
According to  the assumption  \eqref{condition}, we have
\begin{eqnarray}\label{c0-1}
c_0\leq\max\limits_{t\in[0,1]}I(tt_{u^*}u^*)\leq\sup\limits_{t\geq0}I(tu^*)<c(S).
\end{eqnarray}

Next, we are going to prove  the existence of nontrivial weak solutions to problem \eqref{P1}.
When $\mu>0$, according to  \eqref{c0-1} and Lemma \ref{lem-PS-1}, we know that
there exists  a convergent subsequence of $\{u_{n}\}$, still denoted by $\{u_{n}\}$,
such that $u_{n}\rightarrow u$ in $H_0^2(\Omega)$ as $n\rightarrow\infty$,
which implies that $I(u)=c_0$ and $I'(u)=0$, i.e.,
$u$ is a mountain pass type solution to problem \eqref{P1}. When $\mu<0$,
following from \eqref{c0-1} and Lemma \ref{lem-PS-2},  it is easy to see that
problem \eqref{P1} possesses a  nontrivial weak solution. The proof is complete.
\end{proof}

\par
\section{Proofs of the main results}
\setcounter{equation}{0}

In view of Lemma \ref{crucial condition}, one sees that it is sufficient to find $u^*\in H_0^2(\Omega)$
satisfying \eqref{condition} to show the existence of weak solutions to problem \eqref{P1}.
Inspired by Brezis and Nirenberg \cite{Brezis} and Deng et al. \cite{Deng}, we shall do this
with the help of the truncated Talenti functions. Without loss of generality,
we may assume that $0\in \Omega$ is the geometric center of $\Omega$.
Thus, $\rho_{max}=dist(0,\partial\Omega)$. For any $\varepsilon>0$, define
\begin{eqnarray}\label{Talenti-U}
U_\varepsilon(x)=C_N\dfrac{\varepsilon^{\frac{N-4}{2}}}
{(\varepsilon^2+|x|^2)^{\frac{N-4}{2}}}, \qquad x \in \mathbb{R}^N, 
\end{eqnarray}
where $C_N=[N(N-4)(N^2-4)]^{\frac{N-4}{8}}$.
Then $U_\varepsilon(x)$ is a solution of the critical problem
\begin{eqnarray*}
\Delta^2u=u^{2^{**}-1},\qquad x\in \mathbb{R}^N,\ \ \ N\geq5,
\end{eqnarray*}
and $\|U_\varepsilon\|^2=\|U_\varepsilon\|_{2^{**}}^{2^{**}}=S^{\frac{N}{4}}$,
where $S=\inf\limits_{u\in H_0^2(\Omega)\backslash\{0\}}\dfrac{\|u\|^2}{\|u\|_{2^{**}}^{2}}
=\dfrac{\|U_\varepsilon\|^2}{\|U_\varepsilon\|_{2^{**}}^{2}}$.


We have the following estimates on the Talenti functions with a cut-off function
(see \cite{YGu}).
\begin{lemma}\label{lem-Talenti-u}
Assume that $N\geq5$. Let $\varphi\in C_0^\infty(\Omega)$ be a cut-off function such that $\varphi(x)=\varphi(|x|)$, $0\leq\varphi(x)\leq1$ for $x\in \Omega$, and
\begin{eqnarray*}
\varphi(x)=
\begin{cases}
1,&|x|<\rho,\\
0,&|x|> 2\rho,
\end{cases}
\end{eqnarray*}
where $\rho>0$ is a constant such that $B_{2\rho}(0)\subset\Omega$.
Set $u_\varepsilon(x)=\varphi(x)U_\varepsilon(x)$. Suppose that $\varepsilon\rightarrow0$. Then
\allowdisplaybreaks  \begin{align*}
&\|u_\varepsilon\|^2=S^{\frac{N}{4}}+O(\varepsilon^{N-4}),\\
&\|u_\varepsilon\|_{2^{**}}^{2^{**}}=S^{\frac{N}{4}}+O(\varepsilon^{N}).
\end{align*}
\end{lemma}

In what follows, by carefully estimating each term of $\psi_{u_{\varepsilon}}(t)$ we shall verify that \eqref{condition}
holds with $u^*=u_{\varepsilon}$ for suitable $\varepsilon$. The discussion will be divided into three cases, i.e.,
$N\geq9$, $N=8$ and $N=5,6,7$.

\subsection{The case $N\geq9$.}

\begin{lemma}\label{lem-N>8-Talenti}
If $N\geq9$, then, as $\varepsilon\rightarrow0$,
\begin{equation}\label{Talenti-2-9}
\|u_{\varepsilon}\|_{2}^{2}=
C_{N}^{2}\varepsilon^{4}\int_{\mathbb{R}^N}\frac{1}{(1+|y|^{2})^{N-4}}\mathrm{d}y+O(\varepsilon^{N-4}),
\end{equation}
and
\begin{equation}\label{Talenti-dui-9}
\int_{\Omega}u_{\varepsilon}^{2}\ln u_{\varepsilon}^{2}\mathrm{d}x=(N-4)C_N^{2}\varepsilon^{4}\ln \frac{1}{\varepsilon}\int_{\mathbb{R}^{N}}\dfrac{1}
{(1+|y|^{2})^{N-4}}\mathrm{d}y+O(\varepsilon^{4}).
\end{equation}
\end{lemma}

\begin{proof}
Using \eqref{Talenti-U}, the properties of the cut-off function $\varphi$,
change of variable, and polar coordinates transformation, we have
\allowdisplaybreaks
\begin{align*}
\|u_{\varepsilon}\|_2^{2}
&=\int_{\Omega}\varphi^{2}U_{\varepsilon}^{2}\mathrm{d}x\\
&=
C_{N}^{2}\int_{B_{\rho}}\dfrac{\varepsilon^{N-4}}
{(\varepsilon^2+|x|^2)^{N-4}}\mathrm{d}x
+
C_{N}^{2}\int_{B_{2\rho}\backslash B_{\rho}}\varphi^{2}\dfrac{\varepsilon^{N-4}}
{(\varepsilon^2+|x|^2)^{N-4}}\mathrm{d}x\\
&=
C_{N}^{2}\int_{B_{\frac{\rho}{\varepsilon}}}\dfrac{\varepsilon^{N-4}}
{\varepsilon^{2(N-4)}(1+|y|^2)^{N-4}}\varepsilon^{N}\mathrm{d}y
+
O(\varepsilon^{N-4})\\
&=
C_{N}^{2}\varepsilon^{4}\left(\int_{\mathbb{R}^{N}}-\int_{B_{\frac{\rho}{\varepsilon}}}\right)\dfrac{1}
{(1+|y|^2)^{N-4}}\mathrm{d}y
+
O(\varepsilon^{N-4})\\
&=
C_{N}^{2}\varepsilon^{4}\int_{\mathbb{R}^{N}}\dfrac{1}
{(1+|y|^2)^{N-4}}\mathrm{d}y
+O(\varepsilon^{N-4}),
\end{align*}
where we have used the fact that
 \allowdisplaybreaks  \begin{align*}
\int_{B_{2\rho}\backslash B_{\rho}}\varphi^{2}\dfrac{\varepsilon^{N-4}}
{(\varepsilon^2+|x|^2)^{N-4}}\mathrm{d}x
\leq
\varepsilon^{N-4}\int_{B_{2\rho}\backslash B_{\rho}}\dfrac{1}
{|x|^{2(N-4)}}\mathrm{d}x
=O(\varepsilon^{N-4}),
\end{align*}
and
\allowdisplaybreaks
\begin{align}
\label{RN-9-1}\int_{B^{c}_{\frac{\rho}{\varepsilon}}}\dfrac{1}
{(1+|y|^{2})^{N-4}}\mathrm{d}y
=
\omega_{N}\int_{\frac{\rho}{\varepsilon}}^{+\infty}\dfrac{r^{N-1}}
{(1+r^{2})^{N-4}}\mathrm{d}r
\leq
\omega_{N}\int_{\frac{\rho}{\varepsilon}}^{+\infty}\frac{1}{r^{N-7}}
\mathrm{d}r
=O(\varepsilon^{N-8}).
\end{align}
This shows \eqref{Talenti-2-9}.

Next we prove \eqref{Talenti-dui-9}.
According to the definition of $u_{\varepsilon}$, we have
\begin{eqnarray}\label{dui-9}
\int_{\Omega}u_{\varepsilon}^{2}\ln u_{\varepsilon}^{2}\mathrm{d}x
=\int_{\Omega}\varphi^{2}U_{\varepsilon}^{2}\ln (\varphi^{2}U_{\varepsilon}^{2})\mathrm{d}x
=\int_{\Omega}U_{\varepsilon}^{2}\varphi^{2}\ln \varphi^{2}\mathrm{d}x
+\int_{\Omega}\varphi^{2}U_{\varepsilon}^{2}\ln U_{\varepsilon}^{2}\mathrm{d}x
:=I+II.
\end{eqnarray}
To calculate $I$ and $II$, we first estimate $\int_{B_{2\rho}\backslash B_{\rho}}U_{\varepsilon}^{2k}\mathrm{d}x$
with $\frac{N}{2(N-4)}<k<1$ and $k\geq1$, respectively.
When $\frac{N}{2(N-4)}<k<1$, one has
\allowdisplaybreaks
\begin{align}
&\ \ \ \int_{B_{2\rho}\backslash B_{\rho}}U_{\varepsilon}^{2k}\mathrm{d}x
=C_N^{2k}\int_{B_{2\rho}\backslash B_{\rho}}\dfrac{\varepsilon^{k(N-4)}}
{(\varepsilon^2+|x|^2)^{k(N-4)}}\mathrm{d}x\nonumber\\
&
=C_N^{2k}\omega_N\varepsilon^{N-k(N-4)}\int_{\frac{\rho}{\varepsilon}}^{\frac{2\rho}{\varepsilon}}
\dfrac{ y^{N-1}}
{(1+y^2)^{k(N-4)}}\mathrm{d}y
\leq C_N^{2k}\omega_N\varepsilon^{N-k(N-4)}\int_{0}^{+\infty}
\dfrac{ y^{N-1}}
{(1+y^2)^{k(N-4)}}\mathrm{d}y\nonumber\\
\label{Talenti-k-1-9}&=O(\varepsilon^{4}).
\end{align}
When $k\geq1$, we deduce that
\allowdisplaybreaks
\begin{align}
\int_{B_{2\rho}\backslash B_{\rho}}U_{\varepsilon}^{2k}\mathrm{d}x
=C_N^{2k}\int_{B_{2\rho}\backslash B_{\rho}}\dfrac{\varepsilon^{k(N-4)}}
{(\varepsilon^2+|x|^2)^{k(N-4)}}\mathrm{d}x
&\leq C_N^{2k}\varepsilon^{k(N-4)}\int_{B_{2\rho}\backslash B_{\rho}}\dfrac{1}
{|x|^{2k(N-4)}}\mathrm{d}x\nonumber\\
\label{Talenti-k-2-9}&=O(\varepsilon^{4}).
\end{align}

By recalling the fact that $0\leq\varphi\leq1$ with \eqref{logarithmic-1} and
applying \eqref{Talenti-k-2-9} with $k=1$, we have
\begin{eqnarray}\label{dui-I-9}
|I|
\leq \frac{1}{e}\int_{B_{2\rho}\backslash B_{\rho}}U_{\varepsilon}^{2}\mathrm{d}x
=O(\varepsilon^{4}).
\end{eqnarray}
Set
\begin{eqnarray}\label{dui-II-9}
II
=\int_{B_{\rho}}U_{\varepsilon}^{2}\ln U_{\varepsilon}^{2}\mathrm{d}x
+\int_{B_{2\rho}\backslash B_{\rho}}\varphi^{2}U_{\varepsilon}^{2}\ln U_{\varepsilon}^{2}\mathrm{d}x
:=II_{1}+II_{2}.
\end{eqnarray}
 Using \eqref{logarithmic-2} with $\delta>0$ satisfying 
  $\frac{N}{2(N-4)}<1-\delta<1$, we have
\begin{eqnarray}\label{dui-II-2-9}
|II_{2}|
\leq
\int_{B_{2\rho}\backslash B_{\rho}}|U_{\varepsilon}^{2}\ln U_{\varepsilon}^{2}|\mathrm{d}x
\leq
C_{\alpha,\delta}\int_{B_{2\rho}\backslash B_{\rho}}(U_{\varepsilon}^{2(1+\alpha)}+ U_{\varepsilon}^{2(1-\delta)})\mathrm{d}x
=O(\varepsilon^{4}),
\end{eqnarray}
where the last equality is justified by \eqref{Talenti-k-1-9} and \eqref{Talenti-k-2-9}.
In addition
,
one has
\allowdisplaybreaks
\begin{align}
II_{1}
&=C_N^{2}
\int_{B_{\rho}}\dfrac{\varepsilon^{N-4}}
{(\varepsilon^2+|x|^2)^{N-4}}\ln \dfrac{C_N^{2}\varepsilon^{N-4}}
{(\varepsilon^2+|x|^2)^{N-4}}\mathrm{d}x\nonumber\\
&=C_N^{2}
\int_{B_{\frac{\rho}{\varepsilon}}}\dfrac{\varepsilon^{4}}
{(1+|y|^{2})^{N-4}}\ln \dfrac{C_N^{2}}
{\varepsilon^{N-4}(1+|y|^{2})^{N-4}}\mathrm{d}y\nonumber\\
&=
(N-4)C_N^{2}\varepsilon^{4}\ln \frac{1}{\varepsilon}
\int_{B_{\frac{\rho}{\varepsilon}}}\dfrac{1}
{(1+|y|^{2})^{N-4}}\mathrm{d}y+
C_N^{2}\varepsilon^{4}
\int_{B_{\frac{\rho}{\varepsilon}}}\dfrac{1}
{(1+|y|^{2})^{N-4}}\ln \dfrac{C_N^{2}}
{(1+|y|^{2})^{N-4}}\mathrm{d}y
\nonumber\\
&=
(N-4)C_N^{2}\varepsilon^{4}\ln \frac{1}{\varepsilon}
\left(\int_{\mathbb{R}^{N}}-\int_{B^{c}_{\frac{\rho}{\varepsilon}}}\right)\dfrac{1}
{(1+|y|^{2})^{N-4}}\mathrm{d}y\nonumber\\
&\ \ \ +
C_N^{2}\varepsilon^{4}
\int_{B_{\frac{\rho}{\varepsilon}}}\dfrac{1}
{(1+|y|^{2})^{N-4}}\ln \dfrac{C_N^{2}}
{(1+|y|^{2})^{N-4}}\mathrm{d}y\nonumber\\
\label{dui-II-1-9}
&=
(N-4)C_N^{2}\varepsilon^{4}\ln \frac{1}{\varepsilon}\int_{\mathbb{R}^{N}}\dfrac{1}
{(1+|y|^{2})^{N-4}}\mathrm{d}y
+O(\varepsilon^{4}),
\end{align}
where we have used \eqref{RN-9-1} and the fact that
\allowdisplaybreaks  \begin{align*}
&\ \ \ \ |\int_{B_{\frac{\rho}{\varepsilon}}}\dfrac{1}
{(1+|y|^{2})^{N-4}}\ln \dfrac{C_N^{2}}
{(1+|y|^{2})^{N-4}}\mathrm{d}y|\\
&\leq
|\ln C_N^{2}|\int_{B_{\frac{\rho}{\varepsilon}}}\dfrac{1}
{(1+|y|^{2})^{N-4}}\mathrm{d}y
+
(N-4)\int_{B_{\frac{\rho}{\varepsilon}}}\dfrac{1}
{(1+|y|^{2})^{N-4}}\ln (1+|y|^{2})\mathrm{d}y\\
&\leq
|\ln C_N^{2}|\int_{\mathbb{R}^{N}}\dfrac{1}
{(1+|y|^{2})^{N-4}}\mathrm{d}y
+
(N-4)\frac{4}{e}\int_{\mathbb{R}^{N}}\dfrac{1}
{(1+|y|^{2})^{N-4-\frac{1}{4}}}\mathrm{d}y\\
&\leq C.
 \end{align*}
Notice that we have used \eqref{logarithmic-3} with $\sigma=\frac{1}{4}$
in the last but one inequality. Therefore, in accordance with \eqref{dui-9}, \eqref{dui-I-9}, \eqref{dui-II-9}, \eqref{dui-II-2-9} and \eqref{dui-II-1-9}, one obtains \eqref{Talenti-dui-9}. The proof is complete.
\end{proof}

Then, we verify condition \eqref{condition} for $N\geq 9$ with the help of Lemma \ref{lem-N>8-Talenti}.

\begin{lemma}\label{lem-N=9-C}
If $N\geq9$, $\mu>0$ and $\lambda\in \mathbb{R}$, then condition \eqref{condition} holds.
\end{lemma}

\begin{proof}
Since $\mu>0$, for $u_{\varepsilon}$ given in Lemma \ref{lem-Talenti-u},
one has $\lim\limits_{t\rightarrow0}\psi_{u_{\varepsilon}}(t)=0$, $\psi_{u_{\varepsilon}}(t)>0$
for $t>0$ suitably small and $\lim\limits_{t\rightarrow+\infty}\psi_{u_{\varepsilon}}(t)=-\infty$
uniformly for $\varepsilon\in(0,\varepsilon_1)$, where $\varepsilon_1>0$ is a suitably small but fixed number.
Thus, for each $\varepsilon\in(0,\varepsilon_1)$, there exists a $t_\varepsilon\in (0,+\infty)$ such that
\begin{eqnarray*}\label{}
\sup_{t\geq0}I(tu_{\varepsilon})=\sup_{t\geq0}\psi_{u_{\varepsilon}}(t)
=\psi_{u_{\varepsilon}}(t_{\varepsilon}),
\end{eqnarray*}
and
\begin{eqnarray*}\label{}
0=\psi'_{u_{\varepsilon}}\big(t_{\varepsilon})
=t_{\varepsilon}(\| u_{\varepsilon}\|^{2}-\lambda \|u_{\varepsilon}\|_{2}^{2}-\mu \int_{\Omega}u_{\varepsilon}^{2}\ln (t_{\varepsilon}u_{\varepsilon})^{2}\mathrm{d}x-t_{\varepsilon}^{2^{**}-2}
\|u_{\varepsilon}\|_{2^{**}}^{2^{**}}\big),
\end{eqnarray*}
that is
\begin{eqnarray}\label{max-t}
\| u_{\varepsilon}\|^{2}-\lambda \|u_{\varepsilon}\|_{2}^{2}-\mu \int_{\Omega}u_{\varepsilon}^{2}\ln u_{\varepsilon}^{2}\mathrm{d}x
=
t_{\varepsilon}^{2^{**}-2}
\|u_{\varepsilon}\|_{2^{**}}^{2^{**}}+
\mu \|u_{\varepsilon}\|_{2}^{2}\ln t_{\varepsilon}^{2}.
\end{eqnarray}

We first show that there exist $0<T_0<T^0$ such that $T_0\leq t_{\varepsilon}\leq T^0$ for all suitably small $\varepsilon$.
From \eqref{max-t}, Lemmas \ref{lem-Talenti-u} and \ref{lem-N>8-Talenti}, one has, for suitably small $\varepsilon$,
$$ 2S^{\frac{N}{4}}\geq \frac{1}{2}S^{\frac{N}{4}}t_{\varepsilon}^{2^{**}-2}-C|\ln t_{\varepsilon}^{2}|,
$$
which implies that $t_{\varepsilon}$ is bounded from above.
On the other hand, by \eqref{max-t}, Lemmas \ref{lem-Talenti-u}, \ref{lem-N>8-Talenti}
and making use of \eqref{logarithmic-3} with $\sigma=1$, one also has
$$
\frac{1}{2}S^{\frac{N}{4}}\leq2S^{\frac{N}{4}}t_{\varepsilon}^{2^{**}-2}+C t_{\varepsilon}^{2},
$$
which implies that $t_{\varepsilon}$ is bounded from below by some positive constant.

Set
\begin{eqnarray}\label{h-define}
h(t):=\frac{1}{2}t^2\| u_{\varepsilon}\|^{2}-\frac{1}{2^{**}}t^{2^{**}}
\|u_{\varepsilon}\|_{2^{**}}^{2^{**}},\qquad t>0.
\end{eqnarray}
Direct computation shows that $h$ takes its maximum at $t_\varepsilon^*:=\left(\frac{\|u_\varepsilon\|^{2}}{\|u_\varepsilon\|
_{2^{**}}^{2^{**}}}\right)^{\frac{N-4}{8}}$ and by Lemma \ref{lem-Talenti-u}, one has
\begin{eqnarray}\label{h-max}
h(t_\varepsilon^*)=\dfrac{2}{N}\left(\frac{\|u_\varepsilon\|}
{\|u_\varepsilon\|_{2^{**}}}\right)^{\frac{N}{2}}=\frac{2}{N}S^{\frac{N}{4}}
+O(\varepsilon^{N-4}), \qquad \varepsilon\rightarrow0.
\end{eqnarray}
Thus, in view of \eqref{h-max}, Lemma \ref{lem-N>8-Talenti} and recalling
$T_0<t_{\varepsilon}<T^0$ and $\mu>0$, one sees,  for $\varepsilon$ small enough,  that
 \allowdisplaybreaks  \begin{align*}
\sup_{t\geq0}I(tu_{\varepsilon})
&=h(t_{\varepsilon})-\frac{1}{2}\lambda t_{\varepsilon}^2\|u_{\varepsilon}\|_{2}^{2}
+\frac{1}{2}\mu t_{\varepsilon}^2\|u_{\varepsilon}\|_{2}^{2}-\frac{1}{2}\mu \|u_{\varepsilon}\|_{2}^{2}t_{\varepsilon}^2\ln t_{\varepsilon}^2
 -\frac{1}{2}\mu t_{\varepsilon}^2\int_{\Omega}u_{\varepsilon}^{2}\ln u_{\varepsilon}^{2}\mathrm{d}x\\
 &\leq
h(t_{\varepsilon}^{*})-\frac{1}{2}\lambda t_{\varepsilon}^2\|u_{\varepsilon}\|_{2}^{2}
-\frac{1}{2}\mu t_{\varepsilon}^2(\ln t_{\varepsilon}^2-1)\|u_{\varepsilon}\|_{2}^{2}
 -\frac{1}{2}\mu t_{\varepsilon}^2\int_{\Omega}u_{\varepsilon}^{2}\ln u_{\varepsilon}^{2}\mathrm{d}x\\
 &\leq
 \frac{2}{N}S^{\frac{N}{4}}
+O(\varepsilon^{4})
-C \mu \varepsilon^{4}\ln \frac{1}{\varepsilon}\\
&<\frac{2}{N}S^{\frac{N}{4}}.
\end{align*}
Fix such an $\varepsilon>0$ and take $u^*\equiv u_\varepsilon$.
 The proof is complete.
\end{proof}

\subsection{The case $N=8$.}

\begin{lemma}\label{lem-N=8-Talenti}
If $N=8$, then, as $\varepsilon\rightarrow0$,
\begin{eqnarray}\label{Talenti-2-8}
\|u_\varepsilon\|_{2}^{2}=
1920\omega_{8}\varepsilon^{4}\ln\frac{1}{\varepsilon}
+O(\varepsilon^{4}),
\end{eqnarray}
\begin{eqnarray}\label{Talenti-duida-8}
\int_{\Omega}u_{\varepsilon}^{2}\ln u_{\varepsilon}^{2}\mathrm{d}x
\geq1920\omega_{8}\varepsilon^{4}\ln \frac{1}{\varepsilon}\ln\left(\frac{1920(\varepsilon^2+\rho^2)^2}{e^{\frac{47}{3}}
(\varepsilon^2+4\rho^2)^{4}}\right)
+O(\varepsilon^{4}),
\end{eqnarray}
and
\begin{eqnarray}\label{Talenti-duix-8}
\int_{\Omega}u_{\varepsilon}^{2}\ln u_{\varepsilon}^{2}\mathrm{d}x
\leq1920\omega_{8}\varepsilon^{4}\ln \frac{1}{\varepsilon}\ln\left(\frac{1920e^{\frac{37}{3}}(\varepsilon^2+4\rho^2)^2}{
(\varepsilon^2+\rho^2)^{4}}\right)
+O(\varepsilon^{4}).
\end{eqnarray}
\end{lemma}

\begin{proof}
By definition
\allowdisplaybreaks
\begin{align}
\label{T-2-I+II-8}
\|u_{\varepsilon}\|^{2}
=
1920\int_{B_{\rho}}\dfrac{\varepsilon^{4}}
{(\varepsilon^2+|x|^2)^{4}}\mathrm{d}x
+
1920\int_{B_{2\rho}\backslash B_{\rho}}\varphi^{2}\dfrac{\varepsilon^{4}}
{(\varepsilon^2+|x|^2)^{4}}\mathrm{d}x
:=
I
+
II,
\end{align}
where
\begin{eqnarray}\label{T-2-II-8}
|II|\leq
1920\int_{B_{2\rho}\backslash B_{\rho}}\dfrac{\varepsilon^{4}}
{(\varepsilon^2+|x|^2)^{4}}\mathrm{d}x
\leq 1920\varepsilon^{4}\int_{B_{2\rho}\backslash B_{\rho}}\dfrac{1}
{|x|^{8}}\mathrm{d}x
=O(\varepsilon^{4}),
\end{eqnarray}
and
\allowdisplaybreaks
\begin{align}
I&
=
\frac{1920}{2}\omega_{8}\varepsilon^{4}\int_{0}^{\rho}\dfrac{r^{6}}
{(\varepsilon^2+r^2)^{4}}\mathrm{d}(\varepsilon^2+r^2)\nonumber\\
&=
\frac{1920}{2}\omega_{8}\varepsilon^{4}\int_{0}^{\rho}
\left[\dfrac{1}{\varepsilon^2+r^2}
-
\dfrac{3\varepsilon^{2}}
{(\varepsilon^2+r^2)^{2}}
+
\dfrac{3\varepsilon^{4}}
{(\varepsilon^2+r^2)^{3}}
-
\dfrac{\varepsilon^{6}}
{(\varepsilon^2+r^2)^{4}}
\right]\mathrm{d}(\varepsilon^2+r^2)\nonumber\\
&=
\frac{1920}{2}\omega_{8}\varepsilon^{4}
\left[\ln(\varepsilon^2+r^2)
+
\dfrac{3\varepsilon^{2}}
{\varepsilon^2+r^2}
-
\dfrac{3\varepsilon^{4}}
{2(\varepsilon^2+r^2)^{2}}
+
\dfrac{\varepsilon^{6}}
{3(\varepsilon^2+r^2)^{3}}
\right]\bigg|_{0}^{\rho}\nonumber\\
&=
1920\omega_{8}\varepsilon^{4}\ln\frac{1}{\varepsilon}
+
\frac{1920}{2}\omega_{8}\varepsilon^{4}\left[\ln(\varepsilon^2+\rho^2)
+
\dfrac{3\varepsilon^{2}}
{\varepsilon^2+\rho^2}
-\dfrac{3\varepsilon^{4}}
{2(\varepsilon^2+\rho^2)^{2}}
+
\dfrac{\varepsilon^{6}}
{3(\varepsilon^2+\rho^2)^{3}}-\frac{11}{6}
\right]\nonumber\\
\label{T-2-I-8}&=1920\omega_{8}\varepsilon^{4}\ln\frac{1}{\varepsilon}
+O(\varepsilon^{4}).
\end{align}
From  \eqref{T-2-I+II-8}, \eqref{T-2-II-8} and \eqref{T-2-I-8}, we immediately obtain \eqref{Talenti-2-8}.

Now, we are going to  estimate the logarithmic term. Notice that
\allowdisplaybreaks
\begin{align}
&\ \ \ \ \int_{\Omega}u_{\varepsilon}^{2}\ln u_{\varepsilon}^{2}\mathrm{d}x
=
1920\int_{\Omega}\varphi^{2}\dfrac{\varepsilon^{4}}
{(\varepsilon^2+|x|^2)^{4}}\ln \left[1920\varphi^{2}\dfrac{\varepsilon^{4}}
{(\varepsilon^2+|x|^2)^{4}}\right]\mathrm{d}x\nonumber\\
&=
1920\int_{B_{2\rho}}\varphi^{2}\dfrac{\varepsilon^{4}}
{(\varepsilon^2+|x|^2)^{4}}\ln \dfrac{\varepsilon^{4}}
{(\varepsilon^2+|x|^2)^{4}}\mathrm{d}x
+
1920\ln (1920)\int_{B_{2\rho}}\varphi^{2}\dfrac{\varepsilon^{4}}
{(\varepsilon^2+|x|^2)^{4}}\mathrm{d}x\nonumber\\
&\ \ \ +
1920\int_{B_{2\rho}\backslash B_{\rho}}\varphi^{2}\dfrac{\varepsilon^{4}}
{(\varepsilon^2+|x|^2)^{4}}\ln \varphi^{2}\mathrm{d}x
\nonumber\\
\label{dui-8}&:=I_1+I_2+I_3.
\end{align}
Recalling 
\eqref{logarithmic-1}, one has
\begin{eqnarray}\label{dui-I-3-8}
|I_3|\leq\frac{1920}{e}\int_{B_{2\rho}\backslash B_{\rho}}\dfrac{\varepsilon^{4}}
{(\varepsilon^2+|x|^2)^{4}}\mathrm{d}x
\leq\frac{1920}{e}\varepsilon^{4}\int_{B_{2\rho}\backslash B_{\rho}}\dfrac{1}
{|x|^{8}}\mathrm{d}x=O(\varepsilon^4).
\end{eqnarray}
It follows from 
 \eqref{T-2-I-8} and \eqref{dui-I-3-8} that
\allowdisplaybreaks
\begin{align}
I_2
&=
1920\ln (1920)\int_{B_{\rho}}\dfrac{\varepsilon^{4}}
{(\varepsilon^2+|x|^2)^{4}}\mathrm{d}x
+
1920\ln (1920)\int_{B_{2\rho}\backslash B_{\rho}}\varphi^{2}\dfrac{\varepsilon^{4}}
{(\varepsilon^2+|x|^2)^{4}}\mathrm{d}x\nonumber\\
\label{dui-I-2-8}&=1920\ln (1920)\omega_{8}\varepsilon^{4}\ln\frac{1}{\varepsilon}
+O(\varepsilon^{4}).
\end{align}
In addition, by a direct computation, we have
\allowdisplaybreaks
\begin{align}
I_1&=
1920\int_{B_2\rho}\varphi^{2}\dfrac{\varepsilon^{4}}
{(\varepsilon^2+|x|^2)^{4}}\ln \frac{1}{\varepsilon^4\left[1+\left(\frac{|x|}{\varepsilon}\right)^2\right]^4}\mathrm{d}x\nonumber\\
&=4\cdot1920\ln \frac{1}{\varepsilon}\int_{B_2\rho}\varphi^{2}\dfrac{\varepsilon^{4}}
{(\varepsilon^2+|x|^2)^{4}}\mathrm{d}x
 -
4\cdot1920\int_{B_2\rho}\varphi^{2}\dfrac{\varepsilon^{4}}
{(\varepsilon^2+|x|^2)^{4}}\ln \left[1+\left(\frac{|x|}{\varepsilon}\right)^2\right]\mathrm{d}x\nonumber\\
\label{dui-I-1-8}&:=I_{11}+I_{12}.
\end{align}
Applying similar argument to the proof of \eqref{T-2-I-8}, one has
\allowdisplaybreaks
\begin{align}
&\ \ \ \ I_{11}
\geq
4\cdot1920\ln \frac{1}{\varepsilon}\int_{B_{\rho}}\dfrac{\varepsilon^{4}}
{(\varepsilon^2+|x|^2)^{4}}\mathrm{d}x\nonumber\\
&=4\cdot1920\omega_{8}\varepsilon^{4}\ln \frac{1}{\varepsilon}\left\{\ln\frac{1}{\varepsilon} +
\frac{1}{2}\left[\ln(\varepsilon^2+\rho^2)
+
\dfrac{3\varepsilon^{2}}
{\varepsilon^2+\rho^2}
-
\dfrac{3\varepsilon^{4}}
{2(\varepsilon^2+\rho^2)^{2}}
+
\dfrac{\varepsilon^{6}}
{3(\varepsilon^2+\rho^2)^{3}}-\frac{11}{6}
\right]\right\}\nonumber\\
\label{dui-I-11da-8}&=4\cdot1920\omega_{8}\varepsilon^{4}\left(\ln \frac{1}{\varepsilon}\right)^{2}
+
2\cdot1920\omega_{8}\varepsilon^{4}\ln \frac{1}{\varepsilon}\ln\frac{\varepsilon^2+\rho^2}{e^{\frac{11}{6}}}
+O(\varepsilon^{6}\ln \frac{1}{\varepsilon}),
\end{align}
and
\allowdisplaybreaks
\begin{align}
&\ \ \ \ I_{11}
\leq
4\cdot1920\ln \frac{1}{\varepsilon}\int_{B_{2\rho}}\dfrac{\varepsilon^{4}}
{(\varepsilon^2+|x|^2)^{4}}\mathrm{d}x\nonumber\\
&=4\cdot1920\omega_{8}\varepsilon^{4}\ln \frac{1}{\varepsilon}\left\{\ln\frac{1}{\varepsilon}
+
\frac{1}{2}\left[\ln(\varepsilon^2+4\rho^2)
+
\dfrac{3\varepsilon^{2}}
{\varepsilon^2+4\rho^2}
-
\dfrac{3\varepsilon^{4}}
{2(\varepsilon^2+4\rho^2)^{2}}
+
\dfrac{\varepsilon^{6}}
{3(\varepsilon^2+4\rho^2)^{3}}-\frac{11}{6}
\right]\right\}\nonumber\\
\label{dui-I-11x-8}&=4\cdot1920\omega_{8}\varepsilon^{4}\left(\ln \frac{1}{\varepsilon}\right)^{2}
+
2\cdot1920\omega_{8}\varepsilon^{4}\ln \frac{1}{\varepsilon}\ln\frac{\varepsilon^2+4\rho^2}{e^{\frac{11}{6}}}
+O(\varepsilon^{6}\ln \frac{1}{\varepsilon}).
\end{align}
On the other hand,
\allowdisplaybreaks
\begin{align}
I_{12}&\geq-
4\cdot1920\int_{B_2\rho}\dfrac{\varepsilon^{4}}
{(\varepsilon^2+|x|^2)^{4}}\ln \left[1+\left(\frac{|x|}{\varepsilon}\right)^2\right]\mathrm{d}x\nonumber\\
&=
-
2\cdot1920\omega_8\varepsilon^{4}\int_{0}^{\frac{2\rho}{\varepsilon}}\dfrac{r^{6}}
{(1+r^2)^{4}}\ln (1+r^{2})\mathrm{d}(1+r^2)\nonumber\\
&=
-
2\cdot1920\omega_8\varepsilon^{4}\int_{0}^{\frac{2\rho}{\varepsilon}}
\left[\dfrac{1}{1+r^2}
-
\dfrac{3}
{(1+r^2)^{2}}
+
\dfrac{3}
{(1+r^2)^{3}}
-
\dfrac{1}
{(1+r^2)^{4}}
\right]\ln (1+r^{2})\mathrm{d}(1+r^2)\nonumber\\
&\geq
-
2\cdot1920\omega_8\varepsilon^{4}\int_{0}^{\frac{2\rho}{\varepsilon}}
\left[\dfrac{1}{1+r^2}
+\dfrac{3}
{(1+r^2)^{3}}\right]\ln (1+r^{2})\mathrm{d}(1+r^2)\nonumber\\
&\geq
-
2\cdot1920\omega_8\varepsilon^{4}\int_{0}^{\frac{2\rho}{\varepsilon}}
\left[\dfrac{1}{1+r^2}\ln (1+r^{2})
+\dfrac{3}
{1+r^2}\right]\mathrm{d}(1+r^2)\nonumber\\
&=
-
2\cdot1920\omega_8\varepsilon^{4}\left[\frac{1}{2}\big(\ln (1+r^{2})\big)^{2}+3\ln(1+r^{2})\right]\bigg |_{0}^{\frac{2\rho}{\varepsilon}}\nonumber\\
&=
-
2\cdot1920\omega_8\varepsilon^{4}\left\{\frac{1}{2}\left[\ln (\varepsilon^{2}+4\rho^{2})+2\ln\frac{1}{\varepsilon}\right]^{2}
+3\left[\ln (\varepsilon^{2}+4\rho^{2})+2\ln\frac{1}{\varepsilon}\right]
\right\}\nonumber\\
&=
-4\cdot1920\omega_8\varepsilon^{4}\left(\ln\frac{1}{\varepsilon}\right)^{2}
-4\cdot1920\omega_8\varepsilon^{4}\ln\frac{1}{\varepsilon}\ln(\varepsilon^{2}
+4\rho^{2})
-1920\omega_8\varepsilon^{4}\big(\ln(\varepsilon^{2}+4\rho^{2})\big)^{2}\nonumber\\
&\ \ \
-6\cdot1920\omega_8\varepsilon^{4}\ln(\varepsilon^{2}+4\rho^{2})
-12\cdot1920\omega_8\varepsilon^{4}\ln\frac{1}{\varepsilon}\nonumber\\
\label{dui-I-12da-8}&=
-4\cdot1920\omega_8\varepsilon^{4}\left(\ln\frac{1}{\varepsilon}\right)^{2}
-4\cdot1920\omega_8\varepsilon^{4}\ln\frac{1}{\varepsilon}\ln \left[e^{3}(\varepsilon^{2}+4\rho^{2})\right]+O(\varepsilon^{4}),
\end{align}
where we have used \eqref{logarithmic-3} with $\sigma=2$ and the fact that $\frac{1}{2e}<1$.
 Similar to \eqref{dui-I-12da-8}, we have
\allowdisplaybreaks
\begin{align}
I_{12}&\leq-
4\cdot1920\int_{B_\rho}\dfrac{\varepsilon^{4}}
{(\varepsilon^2+|x|^2)^{4}}\ln \left[1+\left(\frac{|x|}{\varepsilon}\right)^2\right]\mathrm{d}x\nonumber\\
&=
-
2\cdot1920\omega_8\varepsilon^{4}\int_{0}^{\frac{\rho}{\varepsilon}}\dfrac{r^{6}}
{(1+r^2)^{4}}\ln (1+r^{2})\mathrm{d}(1+r^2)\nonumber\\
&=
-2\cdot1920\omega_8\varepsilon^{4}\int_{0}^{\frac{\rho}{\varepsilon}}
\left[\dfrac{1}{1+r^2}-\dfrac{3}{(1+r^2)^{2}}+\dfrac{3}{(1+r^2)^{3}}
-\dfrac{1}{(1+r^2)^{4}}\right]\ln (1+r^{2})\mathrm{d}(1+r^2)\nonumber\\
&\leq
-2\cdot1920\omega_8\varepsilon^{4}\int_{0}^{\frac{\rho}{\varepsilon}}
\left[\dfrac{1}{1+r^2}-\dfrac{3}{(1+r^2)^{2}}
-\dfrac{1}{(1+r^2)^{4}}\right]\ln (1+r^{2})\mathrm{d}(1+r^2)\nonumber\\
&\leq-2\cdot1920\omega_8\varepsilon^{4}\int_{0}^{\frac{\rho}{\varepsilon}}
\left[\dfrac{1}{1+r^2}\ln (1+r^{2})
-\dfrac{3}{1+r^2}-\dfrac{1}{1+r^2}\right]\mathrm{d}(1+r^2)\nonumber\\
&=-2\cdot1920\omega_8\varepsilon^{4}\left[\frac{1}{2}\left(\ln (1+r^{2})\right)^{2}-4\ln(1+r^{2})\right]\bigg |_{0}^{\frac{\rho}{\varepsilon}}\nonumber\\
&=-2\cdot1920\omega_8\varepsilon^{4}\left\{\frac{1}{2}\left[\ln (\varepsilon^{2}+\rho^{2})+2\ln\frac{1}{\varepsilon}\right]^{2}
-4\left[\ln (\varepsilon^{2}+\rho^{2})+2\ln\frac{1}{\varepsilon}\right]
\right\}\nonumber\\
&=-4\cdot1920\omega_8\varepsilon^{4}\left(\ln\frac{1}{\varepsilon}\right)^{2}
-4\cdot1920\omega_8\varepsilon^{4}\ln\frac{1}{\varepsilon}\ln(\varepsilon^{2}+\rho^{2})
-1920\omega_8\varepsilon^{4}\left(\ln(\varepsilon^{2}+\rho^{2})\right)^{2}\nonumber\\
&\ \ \
+8\cdot1920\omega_8\varepsilon^{4}\ln(\varepsilon^{2}+\rho^{2})
+16\cdot1920\omega_8\varepsilon^{4}\ln\frac{1}{\varepsilon}\nonumber\\
\label{dui-I-12x-8}&=
-4\cdot1920\omega_8\varepsilon^{4}\left(\ln\frac{1}{\varepsilon}\right)^{2}
-4\cdot1920\omega_8\varepsilon^{4}\ln\frac{1}{\varepsilon}\ln \left(\frac{\varepsilon^{2}+\rho^{2}}{e^{4}}\right)+O(\varepsilon^{4}).
\end{align}
Therefore, from \eqref{dui-8}-\eqref{dui-I-12x-8}, we conclude that
\allowdisplaybreaks
\begin{align*}
\int_{\Omega}u_{\varepsilon}^{2}\ln u_{\varepsilon}^{2}\mathrm{d}x
&\geq
4\cdot1920\omega_{8}\varepsilon^{4}\left(\ln \frac{1}{\varepsilon}\right)^{2}
+2\cdot1920\omega_{8}\varepsilon^{4}\ln \frac{1}{\varepsilon}\ln\frac{\varepsilon^2+\rho^2}{e^{\frac{11}{6}}}
+O(\varepsilon^{6}\ln \frac{1}{\varepsilon})\\
&\ \ \
-4\cdot1920\omega_8\varepsilon^{4}\left(\ln\frac{1}{\varepsilon}\right)^{2}
-4\cdot1920\omega_8\varepsilon^{4}\ln\frac{1}{\varepsilon}\ln \left[e^{3}(\varepsilon^{2}+4\rho^{2})\right]+O(\varepsilon^{4})\\
&\ \ \
+1920\ln (1920)\omega_{8}\varepsilon^{4}\ln\frac{1}{\varepsilon}+O(\varepsilon^{4})\\
&=
1920\omega_{8}\varepsilon^{4}\ln \frac{1}{\varepsilon}\ln\left(\frac{1920(\varepsilon^2+\rho^2)^2}{e^{\frac{47}{3}}
(\varepsilon^2+4\rho^2)^{4}}\right)
+O(\varepsilon^{4}),
\end{align*}
and
\allowdisplaybreaks
\begin{align*}
\int_{\Omega}u_{\varepsilon}^{2}\ln u_{\varepsilon}^{2}\mathrm{d}x
&\leq
4\cdot1920\omega_{8}\varepsilon^{4}\left(\ln \frac{1}{\varepsilon}\right)^{2}
+
2\cdot1920\omega_{8}\varepsilon^{4}\ln \frac{1}{\varepsilon}\ln\frac{\varepsilon^2+4\rho^2}{e^{\frac{11}{6}}}
+O(\varepsilon^{6}\ln \frac{1}{\varepsilon})\nonumber\\
&\ \ \
-4\cdot1920\omega_8\varepsilon^{4}\left(\ln\frac{1}{\varepsilon}\right)^{2}
-4\cdot1920\omega_8\varepsilon^{4}\ln\frac{1}{\varepsilon}\ln \left(\frac{\varepsilon^{2}+\rho^{2}}{e^{4}}\right)+O(\varepsilon^{4})
\nonumber\\
&\ \ \
+1920\ln (1920)\omega_{8}\varepsilon^{4}\ln\frac{1}{\varepsilon}+O(\varepsilon^{4})\nonumber\\
&=
1920\omega_{8}\varepsilon^{4}\ln \frac{1}{\varepsilon}\ln\left(\frac{1920e^{\frac{37}{3}}(\varepsilon^2+4\rho^2)^2}{
(\varepsilon^2+\rho^2)^{4}}\right)
+O(\varepsilon^{4}).
\end{align*}
This completes the proof.
\end{proof}

With Lemma \ref{lem-N=8-Talenti} at hand, we will verify condition \eqref{condition} for $N=8$
under appropriate assumptions on $\lambda,\ \mu$ and $\Omega$.

\begin{lemma}\label{lem-N=8-C}
Assume that $N=8$. If $(\lambda,\mu)\in A$ or $(\lambda,\mu)\in B\cup C$ with $\dfrac{25\cdot1920e^{\frac{\lambda}{\mu}+\frac{34}{3}}}{\rho_{max}^{4}}<1$,
then \eqref{condition} holds.
\end{lemma}

\begin{proof}
As in the proof of Lemma \ref{lem-N=9-C},  one has $\lim\limits_{t\rightarrow0}\psi_{u_{\varepsilon}}(t)=0$ and
$\lim\limits_{t\rightarrow+\infty}\psi_{u_{\varepsilon}}(t)=-\infty$ uniformly for $\varepsilon\in(0,\varepsilon_2)$, where $\varepsilon_2>0$ is a
suitably small but fixed number and $u_{\varepsilon}$ is given in Lemma \ref{lem-Talenti-u}.  This, together with Lemma \ref{lem-mountain pass geometry structure}, implies that there is a $t_{\varepsilon}\in (0,+\infty)$ satisfying \eqref{max-t} and
\begin{eqnarray}\label{N=8-max}
\sup_{t\geq0}I(tu_{\varepsilon})
=h(t_{\varepsilon})-\frac{1}{2}\lambda t_{\varepsilon}^2\|u_{\varepsilon}\|_{2}^{2}
+\frac{1}{2}\mu t_{\varepsilon}^2\|u_{\varepsilon}\|_{2}^{2}-\frac{1}{2}\mu \|u_{\varepsilon}\|_{2}^{2}t_{\varepsilon}^2\ln t_{\varepsilon}^2
 -\frac{1}{2}\mu t_{\varepsilon}^2\int_{\Omega}u_{\varepsilon}^{2}\ln u_{\varepsilon}^{2}\mathrm{d}x,
\end{eqnarray}
where $h$ is defined in \eqref{h-define}. Moreover, there exist $0<T_0<T^0$ such that
\begin{equation}\label{upper-lower-bound}
T_0<t_{\varepsilon}<T^0
\end{equation}
for both $\mu>0$ and $\mu<0$. Indeed,
the boundedness of $t_{\varepsilon}$ for $\mu>0$ and the upper bound of $t_{\varepsilon}$
for $\mu<0$ can be obtained by similar argument to that of the proof of Lemma \ref{lem-N=9-C}.
We only need to show that $t_{\varepsilon}>T_0$ for $\mu<0$. Suppose on the contrary that
there exist a sequence $\{\varepsilon_n\}$ such that $\varepsilon_n\rightarrow0$ and
$t_{\varepsilon_n}\rightarrow0$ as $n\rightarrow\infty$. For each $n$,
take $\gamma(1)=t_{u_{\varepsilon_n}}u_{\varepsilon_n}$ in \eqref{c0},
where $t_{u_{\varepsilon_n}}$ is suitably large such that $I(\gamma(1))<0$.
From this, we have
\begin{equation*}
0<\beta\leq c_{0}\leq I(t_{\varepsilon_n}u_{\varepsilon_n})
\rightarrow0,\qquad n\rightarrow\infty,
\end{equation*}
a contradiction. Thus, the positive lower bound of $t_{\varepsilon}$ follows and \eqref{upper-lower-bound} is valid.

On the basis of \eqref{upper-lower-bound}, we claim that
\begin{eqnarray}\label{inftyxiao}
\|u_{\varepsilon}\|_{2}^{2}t_{\varepsilon}^2\ln t_{\varepsilon}^2
=o(\varepsilon^{4}|\ln \varepsilon|).
\end{eqnarray}
Indeed, from \eqref{Talenti-2-8}, we have
\begin{eqnarray}\label{inftyxiao1}
\|u_{\varepsilon}\|_{2}^{2}
=O(\varepsilon^{4}|\ln \varepsilon|).
\end{eqnarray}
On the other hand, by \eqref{upper-lower-bound}, \eqref{max-t}, Lemmas \ref{lem-Talenti-u} and \ref{lem-N=8-Talenti}, one has
\allowdisplaybreaks  \begin{align*}
t_{\varepsilon}^{2^{**}-2}&=\frac{\| u_{\varepsilon}\|^{2}-\lambda \|u_{\varepsilon}\|_{2}^{2}-\mu \int_{\Omega}u_{\varepsilon}^{2}\ln u_{\varepsilon}^{2}\mathrm{d}x
-
\mu \|u_{\varepsilon}\|_{2}^{2}\ln t_{\varepsilon}^{2}}{\|u_{\varepsilon}\|_{2^{**}}^{2^{**}}}
\\
&=\frac{S^{\frac{N}{4}}+O(\varepsilon^{4}\ln\frac{1}
{\varepsilon})}{S^{\frac{N}{4}}+O(\varepsilon^{8})}\rightarrow1, \qquad \varepsilon\rightarrow0.
\end{align*}
Consequently, $t_{\varepsilon}^{2}\ln t_{\varepsilon}^{2}\rightarrow0$ as $\varepsilon\rightarrow0$.
Combining this with \eqref{inftyxiao1}, one arrives at
\begin{eqnarray*}\label{}
\frac{\|u_{\varepsilon}\|_{2}^{2}t_{\varepsilon}^2\ln t_{\varepsilon}^2}
{\varepsilon^{4}|\ln \varepsilon|}
=\frac{O(\varepsilon^{4}|\ln \varepsilon|)t_{\varepsilon}^2\ln t_{\varepsilon}^2}
{\varepsilon^{4}|\ln \varepsilon|}\rightarrow0,\qquad \varepsilon\rightarrow0,
\end{eqnarray*}
which implies \eqref{inftyxiao}.

In what follows, we divide the proof into two cases.

\noindent{\bf Case 1:} $\mu>0$.

Let $\rho>0$ be such that $\dfrac{384e^{\frac{\lambda}{\mu}-\frac{50}{3}}}{125\rho^{4}}>1$
and take $\varepsilon<\rho$. Following from  \eqref{h-max}, \eqref{Talenti-2-8}, \eqref{Talenti-duida-8},
\eqref{N=8-max}, \eqref{inftyxiao} and the fact that
$$\frac{(\varepsilon^2+\rho^2)^{2}}{(\varepsilon^2+4\rho^2)^{4}}
>\frac{(\rho^2)^{2}}{(\rho^2+4\rho^2)^{4}}=\frac{1}{625\rho^{4}},$$
one obtains
\allowdisplaybreaks
\begin{align*}
\sup_{t\geq0}I(tu_{\varepsilon})
 &\leq
h(t_{\varepsilon}^{*})-\frac{1}{2}\mu t_{\varepsilon}^2
\int_{\Omega}\left[(\frac{\lambda}{\mu}-1)u_{\varepsilon}^{2}+ u_{\varepsilon}^{2}\ln u_{\varepsilon}^{2}\right]\mathrm{d}x
-\frac{1}{2}\mu \|u_{\varepsilon}\|_{2}^{2}t_{\varepsilon}^2\ln t_{\varepsilon}^2\\
 &\leq
 \frac{2}{N}S^{\frac{N}{4}}
+O(\varepsilon^{4})
-\frac{1}{2}\mu t_{\varepsilon}^2
\bigg[(\frac{\lambda}{\mu}-1)1920\omega_{8}\varepsilon^{4}\ln\frac{1}{\varepsilon}\\
&\ \ \ +1920\omega_{8}\varepsilon^{4}\ln \frac{1}{\varepsilon}\ln\left(\frac{1920(\varepsilon^2+\rho^2)^2}{e^{\frac{47}{3}}
(\varepsilon^2+4\rho^2)^{4}}\right)\bigg]+o(\varepsilon^{4}|\ln \varepsilon|)\\
&\leq
 \frac{2}{N}S^{\frac{N}{4}}
+O(\varepsilon^{4})
-\frac{1}{2}\mu t_{\varepsilon}^2
\left[1920\omega_{8}\varepsilon^{4}\ln \frac{1}{\varepsilon}\ln \left(\frac{1920e^{\frac{\lambda}{\mu}-1}(\varepsilon^2+\rho^2)^{2}}{
e^{\frac{47}{3}}(\varepsilon^2+4\rho^2)^{4}}\right)\right]+o(\varepsilon^{4}|\ln \varepsilon|)\\
&\leq
 \frac{2}{N}S^{\frac{N}{4}}
+O(\varepsilon^{4})
-\frac{1}{2}\mu C
\left[1920\omega_{8}\varepsilon^{4}\ln \frac{1}{\varepsilon}\ln \left(\frac{384e^{\frac{\lambda}{\mu}-\frac{50}{3}}}{
125\rho^{4}}\right)\right]+o(\varepsilon^{4}|\ln \varepsilon|).
\end{align*}
Recalling that $\mu>0$ and $\dfrac{384e^{\frac{\lambda}{\mu}-\frac{50}{3}}}{125\rho^{4}}>1$, we know
\begin{equation*}
\sup_{t\geq0}I(tu_{\varepsilon})<\frac{2}{N}S^{\frac{N}{4}},
\end{equation*}
for $\varepsilon$ suitably small.
Fix such an $\varepsilon>0$ and take $u^*\equiv u_\varepsilon$. We obtain \eqref{condition}.

\noindent{\bf Case 2:} $\mu<0$.

For this case we fix $\rho=\rho_{max}$ and let $\varepsilon<\rho$.
Then, in view of \eqref{h-max}, \eqref{Talenti-2-8}, \eqref{Talenti-duix-8}, \eqref{N=8-max}
and \eqref{inftyxiao}, one obtains, for $\varepsilon$ suitably small, that
 \allowdisplaybreaks  \begin{align*}
\sup_{t\geq0}I(tu_{\varepsilon})
 &\leq
h(t_{\varepsilon}^{*})-\frac{1}{2}\mu t_{\varepsilon}^2
\int_{\Omega}\left[(\frac{\lambda}{\mu}-1)u_{\varepsilon}^{2}+ u_{\varepsilon}^{2}\ln u_{\varepsilon}^{2}\right]\mathrm{d}x
-\frac{1}{2}\mu \|u_{\varepsilon}\|_{2}^{2}t_{\varepsilon}^2\ln t_{\varepsilon}^2\\
 &\leq
 \frac{2}{N}S^{\frac{N}{4}}
+O(\varepsilon^{4})+o(\varepsilon^{4}|\ln \varepsilon|)\\
&\ \ \  -\frac{1}{2}\mu t_{\varepsilon}^2
\left[(\frac{\lambda}{\mu}-1)1920\omega_{8}\varepsilon^{4}\ln\frac{1}{\varepsilon}
 +1920\omega_{8}\varepsilon^{4}\ln \frac{1}{\varepsilon}\ln\left(\frac{1920e^{\frac{37}{3}}(\varepsilon^2+4\rho^2)^2}{
(\varepsilon^2+\rho^2)^{4}}\right)\right]\\
&\leq
\frac{2}{N}S^{\frac{N}{4}}
+O(\varepsilon^{4})
-\frac{1}{2}\mu t_{\varepsilon}^2
\left[1920\omega_{8}\varepsilon^{4}\ln \frac{1}{\varepsilon}\ln \left(\frac{1920e^{\frac{\lambda}{\mu}+\frac{34}{3}}(\varepsilon^2+4\rho^2)^{2}}{
(\varepsilon^2+\rho^2)^{4}}\right)\right]+o(\varepsilon^{4}|\ln \varepsilon|)\\
&\leq
 \frac{2}{N}S^{\frac{N}{4}}
+O(\varepsilon^{4})
-\frac{1}{2}\mu C
\left[1920\omega_{8}\varepsilon^{4}\ln \frac{1}{\varepsilon}\ln \left(\frac{25\cdot1920e^{\frac{\lambda}{\mu}+\frac{34}{3}}}{
\rho^{4}}\right)\right]+o(\varepsilon^{4}|\ln \varepsilon|),
\end{align*}
where we have used the fact that $\frac{(\varepsilon^2+4\rho^2)^{2}}{(\varepsilon^2+\rho^2)^{4}}
<\frac{(\rho^2+4\rho^2)^{2}}{(\rho^2)^{4}}=\frac{25}{\rho^{4}}$ and $\mu<0$.
Noticing that $\frac{25\cdot1920e^{\frac{\lambda}{\mu}+\frac{34}{3}}}{\rho_{max}^{4}}<1$,
similar to Case 1, one sees that \eqref{condition} holds with $u^*\equiv u_\varepsilon$ for suitably small $\varepsilon$.
The proof is complete.
\end{proof}

\subsection{The case $N=5,6,7$.}

\begin{lemma}\label{lem-N<8-Talenti}
If $N=5,6,7$, then, as $\varepsilon\rightarrow0$,
\begin{eqnarray}\label{T-2-567}
\|u_{\varepsilon}\|_{2}^{2}=
C_N^2\omega _N \varepsilon^{N-4} \int_{0}^{2\rho}\varphi^{2}r^{7-N}\mathrm{d}r+O(\varepsilon^{N-3}),
\end{eqnarray}
\begin{eqnarray}\label{Talenti-dui-567}
\int_{\Omega}u_{\varepsilon}^{2}\ln u_{\varepsilon}^{2}\mathrm{d}x=(N-4)C_N^{2}\omega_N\varepsilon^{N-4}\ln\varepsilon  \int_{0}^{2\rho}\varphi^{2}r^{7-N}\mathrm{d}r+O(\varepsilon^{N-4}). 
\end{eqnarray}
\end{lemma}

\begin{proof}
We first estimate $\|u_{\varepsilon}\|_{2}^{2}$ with $N=5,6,7$, respectively.
When $N=5$, one has
\allowdisplaybreaks
\begin{align}
&\ \ \ \|u_{\varepsilon}\|_{2}^{2}
=C_5^{2}\int_{B_{2\rho}}\varphi^{2}\dfrac{\varepsilon}
{\varepsilon^2+|x|^2}\mathrm{d}x
=C_5^{2}\omega_{5}\varepsilon\int_{0}^{2\rho}\varphi^{2}r^{2}\dfrac{r^{2}}
{\varepsilon^2+r^2}\mathrm{d}r\nonumber\\
&=C_5^{2}\omega_{5}\varepsilon\int_{0}^{2\rho}\varphi^{2}r^{2}\mathrm{d}r
-
C_5^{2}\omega_{5}\varepsilon\int_{0}^{2\rho}\varphi^{2}r^{2}
\dfrac{\varepsilon^2}
{\varepsilon^2+r^2}\mathrm{d}r\nonumber\\
\label{T-2-5}&:= C_5^{2}\omega_{5}\varepsilon\int_{0}^{2\rho}\varphi^{2}r^{2}\mathrm{d}r
+I,
\end{align}
where
\begin{eqnarray}\label{T-2-I-5}
|I|\leq
C_5^{2}\omega_{5}\varepsilon^{3}\int_{0}^{2\rho}
\dfrac{r^{2}}
{\varepsilon^2+r^2}\mathrm{d}r
\leq
C_5^{2}\omega_{5}\varepsilon^{3}\int_{0}^{2\rho}
\dfrac{r^{2}}
{r^2}\mathrm{d}r
=O(\varepsilon^{3})=O(\varepsilon^{2}).
\end{eqnarray}
When $N=6$, one has
\allowdisplaybreaks
\begin{align}
&\ \ \ \|u_{\varepsilon}\|_{2}^{2}
=C_6^{2}\int_{B_{2\rho}}\varphi^{2}\dfrac{\varepsilon^{2}}
{(\varepsilon^2+|x|^2)^{2}}\mathrm{d}x
=C_6^{2}\omega_{6}\varepsilon^{2}\int_{0}^{2\rho}\varphi^{2}r\dfrac{r^{4}}
{(\varepsilon^2+r^2)^{2}}\mathrm{d}r\nonumber\\
&=C_6^{2}\omega_{6}\varepsilon^{2}\int_{0}^{2\rho}\varphi^{2}r\mathrm{d}r
-
C_6^{2}\omega_{6}\varepsilon^{2}\int_{0}^{2\rho}\varphi^{2}r
\dfrac{\varepsilon^4+2\varepsilon^2r^2}
{(\varepsilon^2+r^2)^{2}}\mathrm{d}r\nonumber\\
\label{T-2-6}&:= C_6^{2}\omega_{6}\varepsilon^{2}\int_{0}^{2\rho}\varphi^{2}r\mathrm{d}r
+II,
\end{align}
where
\allowdisplaybreaks
\begin{align}
|II|&\leq
C_6^{2}\omega_{6}\varepsilon^{2}\int_{0}^{2\rho}
\dfrac{\varepsilon^4r+2\varepsilon^2r^{3}}
{(\varepsilon^2+r^2)^{2}}\mathrm{d}r\nonumber\\
&=\frac{1}{2}C_6^{2}\omega_{6}\varepsilon^6\int_{0}^{2\rho}
\dfrac{1}
{(\varepsilon^2+r^2)^{2}}\mathrm{d}(\varepsilon^2+r^2) +
C_6^{2}\omega_{6}\varepsilon^4\int_{0}^{2\rho}
\left(\dfrac{1}
{\varepsilon^2+r^2}-
\dfrac{\varepsilon^{2}}
{(\varepsilon^2+r^2)^{2}}\right)\mathrm{d}(\varepsilon^2+r^2)\nonumber\\
&=-\frac{1}{2}C_6^{2}\omega_{6}\varepsilon^6
\dfrac{1}
{\varepsilon^2+r^2}\bigg|_{0}^{2\rho}
+
C_6^{2}\omega_{6}\varepsilon^4
\left(\ln (\varepsilon^2+r^2)+
\dfrac{\varepsilon^{2}}
{\varepsilon^2+r^2}\right)\bigg|_{0}^{2\rho}\nonumber\\
&=-\frac{1}{2}C_6^{2}\omega_{6}\varepsilon^6
\dfrac{1}
{\varepsilon^2+4\rho^2}+\frac{1}{2}C_6^{2}\omega_{6}\varepsilon^{4}
+
C_6^{2}\omega_{6}\varepsilon^4
\left(\ln (\varepsilon^2+4\rho^2)+
\dfrac{\varepsilon^{2}}
{\varepsilon^2+4\rho^2}-\ln \varepsilon^2-1\right)\nonumber\\
\label{T-2-II-6}&=O(\varepsilon^{3}).
\end{align}
When $N=7$, direct computation shows that
\allowdisplaybreaks
\begin{align}
&\ \ \ \|u_{\varepsilon}\|_{2}^{2}
=C_7^{2}\int_{B_{2\rho}}\varphi^{2}\dfrac{\varepsilon^{3}}
{(\varepsilon^2+|x|^2)^{3}}\mathrm{d}x
=C_7^{2}\omega_{7}\varepsilon^{3}\int_{0}^{2\rho}\varphi^{2}\dfrac{r^{6}}
{(\varepsilon^2+r^2)^{3}}\mathrm{d}r\nonumber\\
&=C_7^{2}\omega_{7}\varepsilon^{3}\int_{0}^{2\rho}\varphi^{2}\mathrm{d}r
-
C_7^{2}\omega_{7}\varepsilon^{3}\int_{0}^{2\rho}\varphi^{2}
\dfrac{\varepsilon^{6}+3\varepsilon^{4}r^{2}+3\varepsilon^{2}r^{4}}
{(\varepsilon^2+r^2)^{3}}\mathrm{d}r\nonumber\\
\label{T-2-7}&:= C_7^{2}\omega_{7}\varepsilon^{3}\int_{0}^{2\rho}\varphi^{2}\mathrm{d}r
+III,
\end{align}
where
\allowdisplaybreaks
\begin{align}
|III|&\leq C_7^{2}\omega_{7}\varepsilon^{3}\int_{0}^{2\rho}
\dfrac{\varepsilon^{6}+3\varepsilon^{4}r^{2}+3\varepsilon^{2}r^{4}}
{(\varepsilon^2+r^2)^{3}}\mathrm{d}r\nonumber\\
&=C_7^{2}\omega_{7}\varepsilon^{3}\int_{0}^{2\rho}\left[\dfrac{3\varepsilon^{2}}
{\varepsilon^2+r^2}-\dfrac{3\varepsilon^{4}}
{(\varepsilon^2+r^2)^{2}}+\dfrac{\varepsilon^{6}}
{(\varepsilon^2+r^2)^{3}}
\right]\mathrm{d}r\nonumber\\
&=C_7^{2}\omega_{7}\varepsilon^{4}\int_{0}^{\frac{2\rho}{\varepsilon}}
\left[\dfrac{3}
{1+y^2}
+\dfrac{3}
{(1+y^2)^{2}}
+\dfrac{1}
{(1+y^2)^{3}}
\right]\mathrm{d}y\nonumber\\
&\leq C_7^{2}\omega_{7}\varepsilon^{4}\int_{0}^{+\infty}
\left[\dfrac{3}
{1+y^2}
+\dfrac{3}
{(1+y^2)^{2}}
+\dfrac{1}
{(1+y^2)^{3}}
\right]\mathrm{d}y\nonumber\\
\label{T-2-III-7}&=O(\varepsilon^{4}).
\end{align}
Therefore, in accordance with \eqref{T-2-5}-\eqref{T-2-III-7}, we obtain \eqref{T-2-567}.

Next, we consider $\int_{\Omega}u_{\varepsilon}^{2}\ln u_{\varepsilon}^{2}\mathrm{d}x$ with $N=5,6,7$.
One has, for $N=5,6,7$, that
\allowdisplaybreaks
\begin{align}
\int_{\Omega}u_{\varepsilon}^{2}\ln u_{\varepsilon}^{2}\mathrm{d}x
&=C_N^{2}\int_{B_{2\rho}}\varphi^{2}\dfrac{\varepsilon^{N-4}}
{(\varepsilon^2+|x|^2)^{N-4}}\ln \Big(C_N^{2}\varphi^{2}\dfrac{\varepsilon^{N-4}}
{(\varepsilon^2+|x|^2)^{N-4}}\Big) \mathrm{d}x\nonumber\\
&=C_N^{2}\omega_N\int_{0}^{2\rho}\varphi^{2}\dfrac{\varepsilon^{N-4}r^{N-1}}
{(\varepsilon^2+r^2)^{N-4}}\ln \left(C_N^{2}\varphi^{2}\dfrac{\varepsilon^{N-4}}
{(\varepsilon^2+r^2)^{N-4}}\right) \mathrm{d}r\nonumber\\
&=C_N^{2}\ln (C_N^{2})\omega_N\varepsilon^{N-4}\int_{0}^{2\rho}\varphi^{2}\dfrac{r^{N-1}}
{(\varepsilon^2+r^2)^{N-4}}\mathrm{d}r\nonumber\\
&\ \ \ +
(N-4)C_N^{2}\omega_N\varepsilon^{N-4}\ln\varepsilon\int_{0}^{2\rho}\varphi^{2}\dfrac{r^{N-1}}
{(\varepsilon^2+r^2)^{N-4}}\mathrm{d}r\nonumber\\
&\ \ \ +
C_N^{2}\omega_N\varepsilon^{N-4}\int_{0}^{2\rho}\varphi^{2}\dfrac{r^{N-1}}
{(\varepsilon^2+r^2)^{N-4}}\ln \varphi^{2}\mathrm{d}r\nonumber\\
& \ \ \ +
C_N^{2}\omega_N\varepsilon^{N-4}\int_{0}^{2\rho}\varphi^{2}\dfrac{r^{N-1}}
{(\varepsilon^2+r^2)^{N-4}}\ln \dfrac{1}
{(\varepsilon^2+r^2)^{N-4}} \mathrm{d}r\nonumber\\
\label{T-dui-I+I+I+I}&:= I_1+I_2+I_3+I_4,
\end{align}
where
\allowdisplaybreaks
\begin{align}
|I_1|
\leq C_N^{2}\ln (C_N^{2})\omega_N\varepsilon^{N-4}\int_{0}^{2\rho}\dfrac{r^{N-1}}
{r^{2(N-4)}}\mathrm{d}r
&=C_N^{2}\ln (C_N^{2})\omega_N\varepsilon^{N-4}\int_{0}^{2\rho}r^{7-N}
\mathrm{d}r\nonumber\\
\label{T-dui-I-1}&=O(\varepsilon^{N-4}),
\end{align}
and  in virtue of  $\eqref{logarithmic-1}$,
\begin{eqnarray}\label{T-dui-I-3}
|I_3|\leq \frac{1}{e}C_N^{2}\omega_N\varepsilon^{N-4}\int_{0}^{2\rho}\dfrac{r^{N-1}}
{(\varepsilon^2+r^2)^{N-4}}\mathrm{d}r
\leq\frac{1}{e}C_N^{2}\omega_N\varepsilon^{N-4}\int_{0}^{2\rho}\dfrac{r^{N-1}}
{r^{2(N-4)}}\mathrm{d}r=O(\varepsilon^{N-4}).
\end{eqnarray}
According to \eqref{T-2-567},
\begin{eqnarray}\label{T-dui-I-2}
I_2
=(N-4)C_N^{2}\omega_N\varepsilon^{N-4}\ln\varepsilon  \int_{0}^{2\rho}\varphi^{2}r^{7-N}\mathrm{d}r+O(\varepsilon^{N-3}\ln \varepsilon).
\end{eqnarray}
As for $I_4$, we set
\allowdisplaybreaks
\begin{align}
I_4&= C_N^{2}\omega_N\varepsilon^{N-4}\int_{0}^{\rho}\dfrac{r^{N-1}}
{(\varepsilon^2+r^2)^{N-4}}\ln \dfrac{1}
{(\varepsilon^2+r^2)^{N-4}} \mathrm{d}r
\nonumber\\
&\ \ \  +
C_N^{2}\omega_N\varepsilon^{N-4}\int_{\rho}^{2\rho}\varphi^{2}\dfrac{r^{N-1}}
{(\varepsilon^2+r^2)^{N-4}}\ln \dfrac{1}
{(\varepsilon^2+r^2)^{N-4}} \mathrm{d}r\nonumber\\
\label{T-dui-I-4}&:= I_{41}+I_{42}.
\end{align}
Similar to the proof of \eqref{T-2-567}, it is shown that
\allowdisplaybreaks
\begin{align}
I_{41}&=-(N-4) C_N^{2}\omega_N\varepsilon^{N-4}\int_{0}^{\rho}\dfrac{r^{N-1}}
{(\varepsilon^2+r^2)^{N-4}}\ln
(\varepsilon^2+r^2)\mathrm{d}r\nonumber\\
&=-(N-4) C_N^{2}\omega_N\varepsilon^{N-4}\int_{0}^{\rho}r^{7-N}\left(1-
\dfrac{(\varepsilon^2+r^2)^{N-4}-r^{2N-8}}
{(\varepsilon^2+r^2)^{N-4}}\right)\ln
(\varepsilon^2+r^2)\mathrm{d}r\nonumber\\
&=-(N-4) C_N^{2}\omega_N\varepsilon^{N-4}\int_{0}^{\rho}r^{7-N}\ln
(\varepsilon^2+r^2)\mathrm{d}r\nonumber\\
&\ \ \ +(N-4) C_N^{2}\omega_N\varepsilon^{N-4}\int_{0}^{\rho}r^{7-N}\dfrac{(\varepsilon^2+r^2)^{N-4}-r^{2N-8}}
{(\varepsilon^2+r^2)^{N-4}}\ln
(\varepsilon^2+r^2)\mathrm{d}r\nonumber\\
\label{T-dui-I-4-1}&:= I_{411}+I_{412},
\end{align}
where 
\allowdisplaybreaks
\begin{align}
I_{411}
&=-\frac{1}{8-N}(N-4) C_N^{2}\omega_N\varepsilon^{N-4}\int_{0}^{\rho}\ln
(\varepsilon^2+r^2)\mathrm{d}r^{8-N}\nonumber\\
&=-\frac{1}{8-N}(N-4) C_N^{2}\omega_N\varepsilon^{N-4}\left\{\left[ r^{8-N}\ln
(\varepsilon^2+r^2)\right]|_{0}^{\rho}-\int_{0}^{\rho}\frac{2r^{9-N}}
{\varepsilon^2+r^2}\mathrm{d}r\right\}\nonumber\\
&\leq-\frac{1}{8-N}(N-4) C_N^{2}\omega_N\varepsilon^{N-4}\left\{[ r^{8-N}\ln
(\varepsilon^2+r^2)]|_{0}^{\rho}-\int_{0}^{\rho}\frac{2r^{9-N}}
{r^2}\mathrm{d}r\right\}\nonumber\\
&=-\frac{1}{8-N}(N-4) C_N^{2}\omega_N\varepsilon^{N-4}\left\{\rho^{8-N}\ln
(\varepsilon^2+\rho^2)+O(1)\right\}\nonumber\\
\label{T-dui-I-4-11}&=O(\varepsilon^{N-4}).
\end{align}

The estimate of $I_{412}$ will be done separately, depending on $N$. When $N=5$,
by applying \eqref{logarithmic-3} with $\sigma=1$ and using $\dfrac{1}{e}<1$, we have
\allowdisplaybreaks
\begin{align}
|I_{412}|&=
 C_5^{2}\omega_5\varepsilon\int_{0}^{\rho}r^{2}
\dfrac{\varepsilon^2}
{\varepsilon^2+r^2}|\ln
(\varepsilon^2+r^2)|\mathrm{d}r\nonumber\\
&= C_5^{2}\omega_5\varepsilon^4\int_{0}^{\frac{\rho}{\varepsilon}}
\dfrac{ y^{2}}
{1+y^2}\big|\ln
\left(\varepsilon^2(1+y^2)\right)\big|\mathrm{d}y\nonumber\\
&\leq C_5^{2}\omega_5\varepsilon^{4}|\ln
\varepsilon^2|\int_{0}^{\frac{\rho}{\varepsilon}}
\dfrac{y^{2}}
{1+y^2}\mathrm{d}y
+ C_5^{2}\omega_5\varepsilon^{4}\int_{0}^{\frac{\rho}{\varepsilon}}
\dfrac{y^{2}}
{1+y^2}\ln
(1+y^2)\mathrm{d}y\nonumber\\
&\leq C_5^{2}\omega_5\varepsilon^{4}|\ln
\varepsilon^2|\int_{0}^{\frac{\rho}{\varepsilon}}
1\mathrm{d}y
+ C_5^{2}\omega_5\varepsilon^{4}\int_{0}^{\frac{\rho}{\varepsilon}}
y^{2}\mathrm{d}y\nonumber\\
&= C_5^{2}\omega_5\rho\varepsilon^{3}|\ln
\varepsilon^2|
+\frac{1}{3} C_5^{2}\omega_5\rho^{3}\varepsilon
\nonumber\\
\label{T-dui-I-4-112-5}&=O(\varepsilon).
\end{align}

When $N=6$, direct computation yields
\allowdisplaybreaks
\begin{align}
I_{412}&=
2 C_6^{2}\omega_6\varepsilon^{2}\int_{0}^{\rho}
\dfrac{\varepsilon^4r+2\varepsilon^2r^{3}}
{(\varepsilon^2+r^2)^{2}}\ln
(\varepsilon^2+r^2)\mathrm{d}r\nonumber\\
&= C_6^{2}\omega_6\varepsilon^{2}\int_{0}^{\rho}
\left[\dfrac{\varepsilon^4}
{(\varepsilon^2+r^2)^{2}}+\dfrac{2\varepsilon^2r^{2}}
{(\varepsilon^2+r^2)^{2}}\right]\ln
(\varepsilon^2+r^2)\mathrm{d}(\varepsilon^2+r^2)\nonumber\\
&= C_6^{2}\omega_6\varepsilon^{6}\int_{0}^{\rho}
\dfrac{1}
{(\varepsilon^2+r^2)^{2}}\ln
(\varepsilon^2+r^2)\mathrm{d}(\varepsilon^2+r^2)\nonumber\\
&\ \ \ +
2 C_6^{2}\omega_6\varepsilon^{4}\int_{0}^{\rho}
\Big[\dfrac{1}
{\varepsilon^2+r^2}
-\dfrac{\varepsilon^2}
{(\varepsilon^2+r^2)^{2}}\Big]\ln
(\varepsilon^2+r^2)\mathrm{d}(\varepsilon^2+r^2)\nonumber\\
&= -C_6^{2}\omega_6\varepsilon^{6}\int_{0}^{\rho}
\dfrac{1}
{(\varepsilon^2+r^2)^{2}}\ln
(\varepsilon^2+r^2)\mathrm{d}(\varepsilon^2+r^2)\nonumber\\
&\ \ \ +
2 C_6^{2}\omega_6\varepsilon^{4}\int_{0}^{\rho}
\dfrac{1}
{\varepsilon^2+r^2}\ln
(\varepsilon^2+r^2)\mathrm{d}(\varepsilon^2+r^2)\nonumber\\
&= -C_6^{2}\omega_6\varepsilon^{6}
\left[-\frac{\ln(\varepsilon^2+r^2)}{\varepsilon^2+r^2}
-\frac{1}{\varepsilon^2+r^2}\right]\bigg|_{0}^{\rho}+ C_6^{2}\omega_6\varepsilon^{4}
\left(\ln(\varepsilon^2+r^2)\right)^{2}\big|_{0}^{\rho}\nonumber\\
&= -C_6^{2}\omega_6\varepsilon^{6}
\left[-\frac{\ln(\varepsilon^2+\rho^2)}{\varepsilon^2+\rho^2}
+\frac{\ln\varepsilon^2}{\varepsilon^2}
-\frac{1}{\varepsilon^2+\rho^2}+\frac{1}{\varepsilon^2}\right]
 + C_6^{2}\omega_6\varepsilon^{4}
\left[\left(\ln(\varepsilon^2+\rho^2)\right)^{2}-(\ln\varepsilon^2)^{2}\right]\nonumber\\
\label{T-dui-I-412-6}&=O(\varepsilon^{2}).
\end{align}

When $N=7$, by applying \eqref{logarithmic-3} with $\sigma=1$, one has
\allowdisplaybreaks
\begin{align}
|I_{412}|&=
3 C_7^{2}\omega_7\varepsilon^{3}\int_{0}^
{\rho}\dfrac{\varepsilon^{6}+3\varepsilon^{4}r^{2}+3\varepsilon^{2}r^{4}
}{(\varepsilon^2+r^2)^{3}}|\ln
(\varepsilon^2+r^2)|\mathrm{d}r\nonumber\\
&=
3 C_7^{2}\omega_7\varepsilon^{4}\int_{0}^{\frac{\rho}{\varepsilon}}
\left[\dfrac{3}
{1+y^2}-\dfrac{3}
{(1+y^2)^{2}}+\dfrac{1}
{(1+y^2)^{3}}
\right]|\ln
(\varepsilon^2(1+y^2))|\mathrm{d}y\nonumber\\
&\leq
3 C_7^{2}\omega_7\varepsilon^{4}|\ln
\varepsilon^2|\int_{0}^{\frac{\rho}{\varepsilon}}
\left[\dfrac{3}
{1+y^2}-\dfrac{3}
{(1+y^2)^{2}}+\dfrac{1}
{(1+y^2)^{3}}
\right]\mathrm{d}y\nonumber\\
&\ \ \ +
3 C_7^{2}\omega_7\varepsilon^{4}\int_{0}^{\frac{\rho}{\varepsilon}}
\left[\dfrac{3}
{1+y^2}-\dfrac{3}
{(1+y^2)^{2}}+\dfrac{1}
{(1+y^2)^{3}}
\right]\ln
(1+y^2)\mathrm{d}y\nonumber\\
&\leq
3 C_7^{2}\omega_7\varepsilon^{4}|\ln
\varepsilon^2|\int_{0}^{+\infty}
\left[\dfrac{3}
{1+y^2}+\dfrac{1}
{(1+y^2)^{3}}
\right]\mathrm{d}y +
3 C_7^{2}\omega_7\varepsilon^{4}\int_{0}^{\frac{\rho}{\varepsilon}}
4\mathrm{d}y\nonumber\\
\label{T-dui-I-412-7}&=O(\varepsilon^{3}).
\end{align}
In conclusion,  from \eqref{T-dui-I+I+I+I}-\eqref{T-dui-I-412-7}, we obtain \eqref{Talenti-dui-567}. The proof is complete.
\end{proof}

Now, we are able to verify condition \eqref{condition} for $N=5,6,7$.

\begin{lemma}\label{lem-N<8-C}
Assume that $N=5,6,7$. If  $(\lambda, \mu)\in B\cup C$, then \eqref{condition} holds.
\end{lemma}

\begin{proof}
The proof is similar to the proof of Lemma \ref{lem-N=8-C} and we only sketch the outline.
Notice that for $\varepsilon$ suitable small, we still have there is a $t_{\varepsilon}\in (0,+\infty)$
satisfying \eqref{N=8-max} and \eqref{upper-lower-bound}.
From this, \eqref{h-max}, and Lemma \ref{lem-N<8-Talenti}, one has
\allowdisplaybreaks
\begin{align*}
\sup_{t\geq0}I(tu_{\varepsilon})
&\leq h(t_{\varepsilon}^{*})-\frac{1}{2}\lambda t_{\varepsilon}^2\|u_{\varepsilon}\|_{2}^{2}
-\frac{1}{2}\mu t_{\varepsilon}^2(\ln t_{\varepsilon}^2+1)\|u_{\varepsilon}\|_{2}^{2}
 -\frac{1}{2}\mu t_{\varepsilon}^2\int_{\Omega}u_{\varepsilon}^{2}\ln u_{\varepsilon}^{2}\mathrm{d}x\\
 &\leq \frac{2}{N}S^{\frac{N}{4}}+O(\varepsilon^{N-4})-C \mu \varepsilon^{N-4}\ln\varepsilon.
\end{align*}
Then the conclusion follows by choosing $u^*\equiv u_\varepsilon$ with suitable small $\varepsilon$.
The proof is complete.
\end{proof}


By virtue of the above lemmas, we can establish the existence results for problem \eqref{P1}.
\begin{proof}[Proof of Theorem \ref{th} ]
$(i)$ When $N\geq8$ and $(\lambda,\mu)\in  A$, following from Lemmas \ref{lem-N=9-C} and  \ref{lem-N=8-C},
we know that condition \eqref{condition} holds. Consequently, there exists a mountain pass type solution
$u$ to problem \eqref{P1} from Lemma \ref{crucial condition}. Furthermore, in view of Remark \ref{Nehari manifold},
we know that $u$ is a ground state solution to problem \eqref{P1}.

$(ii)$   When $N=8$, $(\lambda,\mu)\in B\cup C$ with $\dfrac{25\cdot1920e^{\frac{\lambda}{\mu}+\frac{34}{3}}}{
\rho_{max}^{4}}<1$, it follows from Lemma \ref{lem-N=8-C} that condition \eqref{condition} is valid.
Thus, from Lemma \ref{crucial condition}, we know that there exists a nontrivial weak solution $u$ to problem \eqref{P1}.

$(iii)$ When $N=5,6,7$ and $(\lambda,\mu)\in B\cup C$, by combining Lemma \ref{crucial condition} with Lemma \ref{lem-N<8-C},
we know that there exists a nontrivial weak solution $u$ to problem \eqref{P1}. The proof of Theorem \ref{th} is complete.
\end{proof}

Finally, we give the proof of the nonexistence of positive solutions in the following.
\begin{proof}[Proof of Theorem \ref{th-2}]
To this end, we assume by contradiction that problem \eqref{P1} admits a positive  solution $u_{0}$.
Let $\phi_{1}(x)>0$  be the first eigenfunction corresponding to $\lambda_{1}(\Omega)$.
Then, by integrating by parts, one has
\begin{eqnarray*}\label{}
\int_{\Omega}(\lambda
+\mu \ln u_0^{2}+ u_0^{2^{**}-2})u_0\phi_{1}\mathrm{d}x
&=\int_{\Omega}\Delta u_0\Delta \phi_{1}\mathrm{d}x=\int_{\Omega}\phi_{1}\Delta^{2} u_0 \mathrm{d}x=\lambda_{1}(\Omega)\int_{\Omega} u_0 \phi_{1}\mathrm{d}x,
\end{eqnarray*}
or equivalently
\begin{eqnarray*}\label{}
\int_{\Omega}(\lambda-\lambda_{1}(\Omega)
+\mu \ln u_0^{2}+ u_0^{2^{**}-2})u_0\phi_{1}\mathrm{d}x=0.
\end{eqnarray*}
Set
\begin{eqnarray*}\label{}
f(t):=\lambda-\lambda_{1}(\Omega)
+\mu \ln t^{2}+ t^{2^{**}-2}, \qquad t>0,
\end{eqnarray*}
Direct computation shows that $f$ attains its minimum at $\widetilde{t}:=\left(\frac{-\mu(N-4)}{4}\right)^{\frac{N-4}{8}}$ and
 \allowdisplaybreaks
\begin{align*}
f(t)\geq f(\widetilde{t})= \frac{-\mu(N-4)}{4}+\frac{\mu(N-4)}{4}\ln \left(\frac{-\mu(N-4)}{4}\right)+\lambda-\lambda_{1}(\Omega)\geq0,
\end{align*}
by assumption. Following from  $u_0\in H_0^2(\Omega)$ and $u_{0}>0$, $ \phi_{1}>0$,
we claim that $\int_{\Omega}f(u_0)u_0\phi_{1}\mathrm{d}x>0$. Otherwise, $f(u_{0})=0$ a.e. in $\Omega$, which implies that $u_0=\left(\frac{-\mu(N-4)}{4}\right)^{\frac{N-4}{8}}$ a.e. in $\Omega$, and
this contradicts with the assumption that $u_0\in H_0^2(\Omega)$. Thus, we see that
\allowdisplaybreaks
\begin{align*}
 0=\int_{\Omega}(\lambda-\lambda_{1}(\Omega)+\mu \ln u_0^{2}+ u_0^{2^{**}-2})u_0\phi_{1}\mathrm{d}x=\int_{\Omega}f(u_0)u_0\phi_{1}\mathrm{d}x>0,
\end{align*}
a contradiction. Hence problem \eqref{P1} has no positive solutions.
This completes the proof.

\end{proof}



\begin{thebibliography}{xx}

\bibitem{MMAlGharabli}
M. M. Al-Gharabli, S. A. Messaoudi, Existence and a general decay result for a plate
equation with nonlinear damping and a logarithmic source term, J. Evol. Equ.,  {\bf18}(2018),
\ 105-125.



\bibitem{COAlves}
C. O. Alves, G. M. Figueiredo, On multiplicity and concentration of positive
solutions for a class of quasilinear problems with critical exponential growth in
$\mathbb{R}^N$, J. Differ. Equ., {\bf246}(2009), \ 1288-1311.


\bibitem{LJAn}
L. J. An, Loss of hyperbolicity in elastic-plastic material at finite strains,
SIAM J. Appl. Math., {\bf53}(1993), \ 621-654.


\bibitem{LJAn2}
L. J. An, A. Peirce, The effect of microstructure on elastic-plastic models, SIAM J. Appl.
Math., {\bf54}(1994), \ 708-730.

\bibitem{LJAn3}
L. J. An, A. Peirce, A weakly nonlinear analysis of elasto-plastic-microstructure models,
SIAM J. Appl. Math., {\bf55}(1995), \ 136-155.


\bibitem{BBarrios}
B. Barrios,  E. Colorado, A. de Pablo, U. S\'{a}nchez,  On some critical problems
for the fractional Laplacian operator,  J. Differ. Equ., {\bf252}(2012), \ 6133-
6162.


\bibitem{HBr}
H. Br\'{e}zis, E. Lieb, A relation between pointwise convergence of functions
and convergence of functionals, Proc. Amer. Math. Soc., {\bf88}(1983), \  486-490.



\bibitem{Brezis}
H. Br\'{e}zis, L. Nirenberg, Positive solutions of nonlinear elliptic equations
involving critical Sobolev exponents,  Comm. Pure Appl. Math., {\bf36}(1983), \ 437-477.


\bibitem{LDAmbrosio}
L. D'Ambrosio, E. Jannelli, Nonlinear critical problems for the biharmonic operator
with Hardy potential, Calc. Var. Partial Differential Equations, {\bf54}(2015), \ 365-396.



\bibitem{Deng}
Y. Deng, Q. He, Y. Pan, X. Zhong, The existence of positive solution for an
elliptic problem with critical growth and logarithmic perturbation,
(arXiv: submit/4525584 [math.AP], 4 Oct 2022).


\bibitem{Deng4}
 Y. Deng, Y. Li, Existence and bifurcation of the positive solutions for a semilinear
 equation with critical exponent, J. Differ. Equ., {\bf130}(1996),\  179-200.



\bibitem{Deng2}
Y. Deng, G. Wang, On inhomogeneous biharmonic equations involving
critical exponents, Proc. Roy. Soc. Edinburgh Sect. A,
  {\bf129}(1999), \ 925-946.


\bibitem{Deng3}
Y. Deng, J. Yang, Existence of multiple solutions and bifurcation for critical
semilinear biharmonic equations, Syst. Sci. Math. Sci.,  {\bf8}(1995), \ 319-326.

\bibitem{LGross}
L. Gross,  Logarithmic Sobolev inequalities, Amer. J. Math., {\bf97}(1976), \ 1061-1083.


\bibitem{YGu}
Y. Gu, Y. Deng, X. Wang,  Existence of nontrivial solutions for critical semilinear
biharmonic equations, Syst. Sci. Math. Sci.,  {\bf7}(1994),  \ 140-152.



\bibitem{DNaimen}
D. Naimen, The critical problem of Kirchhoff type elliptic equations in dimension four,
J. Differ. Equ., {\bf257}(2014), \ 1168-1193.



\bibitem{PohozaevSI}
S. I. Pohozaev, Eigenfunctions of the equation $\Delta u+\lambda f(u)=0$,
Soviet Math. Doklady, {\bf6}(1965), \ 1408-1411.



\bibitem{PPucci}
P. Pucci,  J. Serrin,  Critical exponents and critical dimensions for polyharmonic operators,
J. Math. Pures Appl., {\bf69} (1990),\ 55-83.


\bibitem{MWillem}
M. Willem,  Minimax theorems, Birkh\"{a}user, Boston, 1996.


\bibitem{XMingqi}
M. Xiang, V. D. R\u{a}dulescu, B. Zhang, Fractional Kirchhoff problems with critical Trudinger-Moser nonlinearity,
Calc. Var. Partial Differential Equations, {\bf58}(2019), \ 1-27.








\end{thebibliography}
\end{document}